\documentclass[11pt,a4paper,eqno]{amsart}
\usepackage[left=0.9in, right=0.9in, top=1.5in, bottom=1.5in]{geometry}
\usepackage[utf8]{inputenc}
\usepackage[T1]{fontenc} 
\usepackage[english]{babel}
\usepackage{yhmath}
\usepackage{amsmath,amssymb}
\usepackage{stmaryrd}
\usepackage{amsmath}
\usepackage{amsthm}
\usepackage{tikz}
\usepackage{verbatim}

\usepackage{graphicx}
\usepackage{enumitem}
\usepackage{footnote} %
\usepackage{appendix}

\newtheorem{theorem}{Theorem}[section]

\newtheorem{lemma}[theorem]{Lemma}

\newtheorem{proposition}[theorem]{Proposition}
\newtheorem{definition}[theorem]{Definition}
\newtheorem{remark}[theorem]{Remark}
\newtheorem{open}[theorem]{Open Problem}
\newtheorem*{mtheorem1}{Main Theorem 1}
\newtheorem*{mtheorem2}{Main Theorem 2}
\newtheorem*{theorem*}{Theorem}

\newcommand{\R}{\mathbb{R}}

\newcommand{\e}{\epsilon}

\newcommand{\N}{\mathbb{N}}
\renewcommand{\H}{\mathcal{H}}
\renewcommand{\S}{\mathbb{S}}

\newcommand{\ds}{\displaystyle}
\newcommand{\rmnote}[1]{}%{\mnote{#1}}

\usepackage{braket}
\usepackage{esint}
\newcommand{\abs}[1]{{\left|#1\right|}}

\title[A stability result for the first Robin-Neumann eigenvalue]{A stability result for the first Robin-Neumann eigenvalue: \\ a double perturbation approach}
\author{Simone Cito$^1$, Gloria Paoli$^2$, Gianpaolo Piscitelli$^3$ }

\address{$^1$Dipartimento di Matematica e Fisica \lq\lq E. De Giorgi\lq\lq, Università del Salento, Via per Arnesano, 73100 Lecce, Italy.}
\email{simone.cito@unisalento.it}
\address{$^2$ Dipartimento di Matematica e Applicazioni "R. Caccioppoli", Università degli studi di Napoli Federico II, Via Cinthia - Complesso Universitario di Monte Sant'Angelo, 80126 Napoli, Italy.}
\email{gloria.paoli@unina.it}
\address{$^3$Dipartimento di Scienze Economiche, Giuridiche, Informatiche e Motorie, Universit\`a degli Studi di Napoli Parthenope, Via Guglielmo Pepe, Rione Gescal, 80035 Nola (NA), Italy.}
\email{gianpaolo.piscitelli@uniparthenope.it (corresponding author)}

\begin{document}

\maketitle

\begin{abstract}
Let $\Omega=\Omega_0\setminus \overline{\Theta}\subset \mathbb{R}^n$,  $n\geq 2$, where $\Omega_0$ and $\Theta$ are two open, bounded and convex sets such that $\overline{\Theta}\subset \Omega_0$ and let $\beta<0$ be a given parameter. We consider the eigenvalue problem for the Laplace operator associated to $\Omega$, with Robin boundary condition on $\partial \Omega_0$ and Neumann boundary condition on $\partial \Theta$. In \cite{paoli2020sharp} it is proved that  the spherical shell is the only maximizer for the first Robin-Neumann eigenvalue in the class of domains $\Omega$ with  fixed outer perimeter and volume. 

We establish a quantitative version of the afore-mentioned  isoperimetric inequality; the main novelty consists  in the  introduction  of a new type of hybrid asymmetry, that turns out to be the suitable one to treat  the different conditions on the outer and internal boundary. Up to our knowledge, in this context, this is the first stability  result in which \emph{both}  the outer and the inner boundary are perturbed. 

%It is worth noticing that the spectral stability results in literature in which both the inner and the outer boundary are perturbed are in an early stage of development. We introduce a new type of hybrid asymmetry, that turns out to be the suitable one to treat  the different conditions on the two boundaries.

\end{abstract}

\section{Introduction}
Let $\Omega=\Omega_0\setminus \overline{\Theta}$ be a subset of $\mathbb{R}^n$, with $n\geq 2$, where $\Omega_0$ is an open bounded convex set, $\Theta$ is a finite union of open sets homeomorphic to balls such that $\overline{\Theta}\subset \Omega_0$ and let $\beta<0$ be a given parameter.  

In this paper, we study the first Robin-Neumann eigenvalue of the the Laplace operator, defined as:
\begin{equation}
\label{eig_min}		\lambda_1%^{RN}
(\beta, \Omega)=  \min_{\substack{\psi\in H^{1}(\Omega)\\ \psi\not \equiv0}}R_\Omega(\psi), 
\end{equation}
where
\begin{equation}
\label{Raylegh_quotient}	
R_\Omega(\psi)=\dfrac{\ds\int _{\Omega}|\nabla \psi|^2\;dx+\beta\ds\int_{\partial \Omega_0}\psi^2 \;d\mathcal{H}^{n-1}}{\ds\int_{\Omega}\psi^2\;dx}.
\end{equation}
%where we denote by $\Omega=\Omega_0\setminus \overline{\Theta}$ a subset of $\mathbb{R}^n$, $n\geq 2$, such that $\Omega_0$ is an open, bounded and convex set and $\Theta$ is a convex set  %\textcolor{red}{a finite union of sets, each of one  diffeomorphic to a ball of $\mathbb{R}^n$ and with Lipschitz boundary} 
%such that $\overline{\Theta}\subset \Omega_0$, and by  $\beta<0$ the Robin boundary parameter. 
%Let us observe that when $\beta=0$, then $\lambda_1(\beta, \Omega)$ is the first Neumann eigenvalue \textcolor{blue}{non lo scriverei, visto che stiamo considerando i valori negativi di beta}; meanwhile, for $\beta =+\infty$, $\lambda_1(\beta, \Omega)$ asimptotically it gives the first Dirichlet-Neumann boundary eigenvalue. 
If $w\in H^1(\Omega)$ is a minimizer of \eqref{eig_min}, then it satisfies
\begin{equation}
\label{eig_prob}
\begin{cases}
	-\Delta w=\lambda_1%^{RN}
	(\beta, \Omega) 
	w & \mbox{in}\ \Omega\vspace{0.2cm}\\
\dfrac{\partial w}{\partial \nu}+\beta w=0&\mbox{on}\ \partial \Omega_0\vspace{0.2cm}\\ 
\dfrac{\partial w}{\partial \nu}=0&\mbox{on}\ \partial\Theta,
\end{cases}
\end{equation}
where we denote by $ \nu $  the unit  outer normal to the boundary and  $\beta$ is usually called the Robin boundary parameter associated to the problem.

In \cite{giorgi1}, the authors interpret $\lambda_1(\beta, \Omega)$, in the case $\Theta=\emptyset$, as  \emph{the growth rate of particles arising from branching
of Brownian motions with respect to the local time and the splitting rate $-\beta$}. 
In our case, we are considering, in addition to this situation, an insulating internal hole $\Theta$. 
Moreover, for the physical interpretation in the superconductivity context see also \cite{sabina} and the references therein. 

We denote by $B_{R}$ the open ball of $\mathbb{R}^n$ of radius $R$ centered at the origin and we set $A_{R_1,R_2}:=B_{R_2}\setminus\overline{B_{R_1}}$, with $0<R_1<R_2$. Moreover, we denote, as usual, by $P(\cdot)$ and $|\cdot|$,  respectively, the perimeter and the volume of a set.
In \cite{paoli2020sharp} it has been proved that 
%the authors prove that the spherical shell maximizes \eqref{eig_min} among holed domains with both outer perimeter and volume fixed. More precisely, %the result is the following. \begin{theorem*}[Theorem $3.1$ in \cite{paoli2020sharp}]\label{main_thm_PPT}
%it holds that
\begin{equation}
\label{ineq_eig}		\lambda_1(\beta, \Omega)\leq \lambda_1	(\beta, A),
\end{equation}
where $A=A_{R_1,R_2}$ is the spherical shell such that $P(B_{R_2})=P(\Omega_0)$ and $|\Omega|=|A|$ and 
equality  holds if and only if $\Omega=A$. 
%\end{theorem*}

The aim of the present  paper is to prove a quantitative estimate of the isoperimetric inequality \eqref{ineq_eig}.
More precisely, we want to improve \eqref{ineq_eig}, obtaining an inequality of the following type: 
\begin{equation*}
\lambda_1%^{RN}
(\beta, A)-\lambda_1%^{RN}
(\beta, \Omega)\geq g(\alpha(\Omega)),
\end{equation*}
where $g(\cdot)$ is a modulus of continuity, i.e. a positive continuous increasing function, vanishing only at $0$ and $\alpha(\cdot)$ is a suitable  asymmetry functional, i.e. a non negative functional such that $\alpha(\Omega) = 0$ if and only if $\Omega=A$.
%Indeed, the measurement of the deviation of an admissible set from the radial maximal shapes is a natural question arising from finding the optimal shape. 

In order to get this kind of quantitative  result, we need to introduce the following class of sets, in which \eqref{ineq_eig} turns out to be stable. Let us fix  $0<R_1<R_2<+\infty$ and let us  denote by $\vartheta_{R_1,R_2}$ a positive constant such that $\vartheta_{R_1,R_2}\le\min\{R_1,R_2-R_1\}/2$.  We set
\begin{equation}
\label{admissiblesets}
\begin{split}\mathcal T_{R_1,R_2}= \{\Omega\subseteq\R^n :&\ \Omega=\Omega_0\setminus\overline{\Theta},\  \Omega_0,\Theta\ \text{convex},\ |\Omega|=|A_{R_1,R_2}|,\ P(\Omega_0)=P(B_{R_2}),\\ & d_\mathcal{H}(\Theta,\Omega_0)\ge\vartheta_{R_1,R_2}\ \text{and}\ \rho(\Theta)\ge\vartheta_{R_1,R_2}\},
\end{split}
\end{equation}
 where  $\rho(\cdot)$ is the inradius of a set  and $d_{\mathcal H}(\cdot, \cdot )$ is the Hausdorff distance between two convex sets. More precisely, given $E,F\subseteq\mathbb{R}^n$ convex sets, the Hausdorff distance is defined as follows
 \[
d_{\mathcal H}(E,F)=\inf\{\varepsilon>0\ : \ E\subset F+B_\varepsilon, \ F\subset E+B_\varepsilon\}.
\]
Thus, as a direct consequence of \eqref{ineq_eig}, we have that  the spherical shell $A_{R_1,R_2}$ is the only maximizer for  the shape optimization problem
\begin{equation}\label{robinproblem}
\max\left\{\lambda_1(\beta,\Omega):\Omega\in\mathcal{T}_{R_1,R_2}\right\}. 
\end{equation}
In order to study the stability of the reverse Faber-Krahn type inequality \eqref{ineq_eig} in the class $\mathcal{T}_{R_1,R_2}$, we need to introduce a suitable asymmetry functional, that takes into  account  both the perturbations  on the outer and inner boundary with the different boundary conditions. 
%---Once the reverse Faber-Krahn type inequality \eqref{ineq_eig} is set in $\mathcal{T}_{R_1,R_2}$, we need a suitable asymmetry functionals to provide a quantitative version. In order to treat the different behavior of the two boundaries, we introduce an hybrid asymmetry to study together the outer and the inner perturbation of the boundary.

To treat the perturbation of the outer boundary, we use this  standard notion of asymmetry for a convex set:
\begin{equation}
    \label{fraenkel_asy}
\mathcal{A}_\H(\Omega_0)=\inf_{x_0\in\mathbb{R}^n} \left\{\frac{d_{\mathcal{H}}(\Omega_0, B_{R}(x_0))}{\omega_nR^n}\ : \  P(\Omega_0)=P(B_{R}(x_0))\right\},
\end{equation}
where $\omega_n$ is the measure of the unit  ball in $\mathbb{R}^n$.

On the other hand, we also need to take  into account the perturbation of the  inner hole $\Theta$. %Specifically, we measure the deviation from the most "regular" hole with respect to $\Omega_0$: the inner parallel set relative to $\Omega_0$ (that is the set of points in $\Omega_0$ with distance less than a fixed constant) with the same measure as $\Theta$.
In order to do that, we recall the definition of the test function  $w_{\Omega_0}$ used in \cite{paoli2020sharp} to prove \eqref{ineq_eig}:
\begin{equation}
\label{GFerone}
w_{\Omega_0}(x):=
\begin{cases}
G(d(x))\quad & \text{if} \  d(x)< R_2-R_1\\
z_m & \text{if} \ d(x)\geq R_2-R_1
\end{cases} \ x\in\Omega_0,
\end{equation}
where $z$ be the first positive eigenfunction on the spherical shell $A_{R_1,R_2}$ with   $\|z\|_{L^2(A_{R_1,R_2})}=1$, $d(x)$ is the distance function  from the  boundary of $\Omega_0$ and $G$ is defined as
\begin{equation*}
%\label{G_t_max}
G^{-1}(t)=\int_{t}^{z_M}\dfrac{1}{\ell(\tau)}\;d\tau,
\end{equation*}
with  $\ell(t)=|\nabla z|_{z=t}$, determined for $z_m<t\leq z_M$ with $z_m:=\min_A z$  and $z_M:=\max_A z$.
We point out that $w_{\Omega_0}$ is  a map depending only on the distance from the outer boundary (see Section \ref{Robin_subsec} for further details) and, moreover, it  is constant on the boundaries of the inner parallels to $\Omega_0$. This type of functions are usually called \textit{web-function} and this name  appeared  for the first time in \cite{gazzola1999existence}, since their level lines recall the web of a spider.

%We remark taht Gazzola \cite{gazzola1999existence} firstly gave this name because their level lines recall a spider’s web. 

Therefore, $w_{\Omega_0}$ can be used as a test function for every admissible holed set with outer box $\Omega_0$.
Now, we are in position to give the following definition (see also Figure \ref{fig_01_hybrid}). 
\begin{definition}[Hybrid asymmetry]
\label{inner_asymmetry_def}
We define the hybrid asymmetry of $\Omega=\Omega_0\setminus\overline{\Theta}$ as:
\begin{equation} \label{hybrid}
\alpha(\Omega):=\max\left\{g(\mathcal{A}_\H(\Omega_0)),\tilde{\mathcal{A}}(\Theta;\Omega_0)\right\}.
\end{equation}
Here,  $\mathcal{A}_\mathcal{H}(\cdot)$ is the Hausdorff asymmetry defined in \eqref{fraenkel_asy}  and $g(\cdot)$ is the positive increasing function defined as
\begin{equation*}
g(s)=\begin{cases}
s^2 &\text{if $n=2$}\\
f^{-1}(s^2) &\text{if $n=3$}\\
s^\frac{n+1}{2} &\text{if $n\ge 4$},
\end{cases}
\end{equation*}
with $f(t)=\sqrt{t \log \left(\frac{1}{t}\right)}$ for $0<t<e^{-1}$. 

The term $\tilde{\mathcal{A}}(\Theta; \Omega_0)$ is the weak weighted Fraenkel-type asymmetry of  $\Theta$ relative to the set $\Omega_0$, defined as  
\begin{equation}\label{eq:atilde}
\tilde{\mathcal{A}}(\Theta;\Omega_0):=\int_{\Theta\setminus K_{\Omega_0} 
}(|\nabla w_{\Omega_0}|^2+w_{\Omega_0}^2(x)-z_m^2)\:dx, 
\end{equation}
where 
\begin{itemize}
\item $w_{\Omega_0}\in H^1(\Omega_0)$ is  the test function defined in \eqref{GFerone};% (i.e. $\|z\|_{L^2(A_{R_1,R_2})}=1$);
\item $K_{\Omega_0}$ is the inner parallel set relative to $\Omega_0$, i.e. the set 
\begin{equation*}
K_{\Omega_0}
=\{ x\in\Omega_0\;:\; d(x)>t_{\Omega_0}  \},
\end{equation*}
with $t_{\Omega_0}\in(0,\rho (\Omega_0)) $ such that $|K_{\Omega_0}
|=|\Theta|$;
\item $z$ is the first positive eigenfunction of \eqref{eig_min} on the $A_{R_1,R_2}$ with  $\|z\|_{L^2(A_{R_1,R_2})}=1$;
\item $z_m$ is the minimum of $z$ on $A_{R_1,R_2}$.
\end{itemize}
\end{definition}

\begin{figure}[!ht]
    \centering    \includegraphics[width=.3312\textwidth]{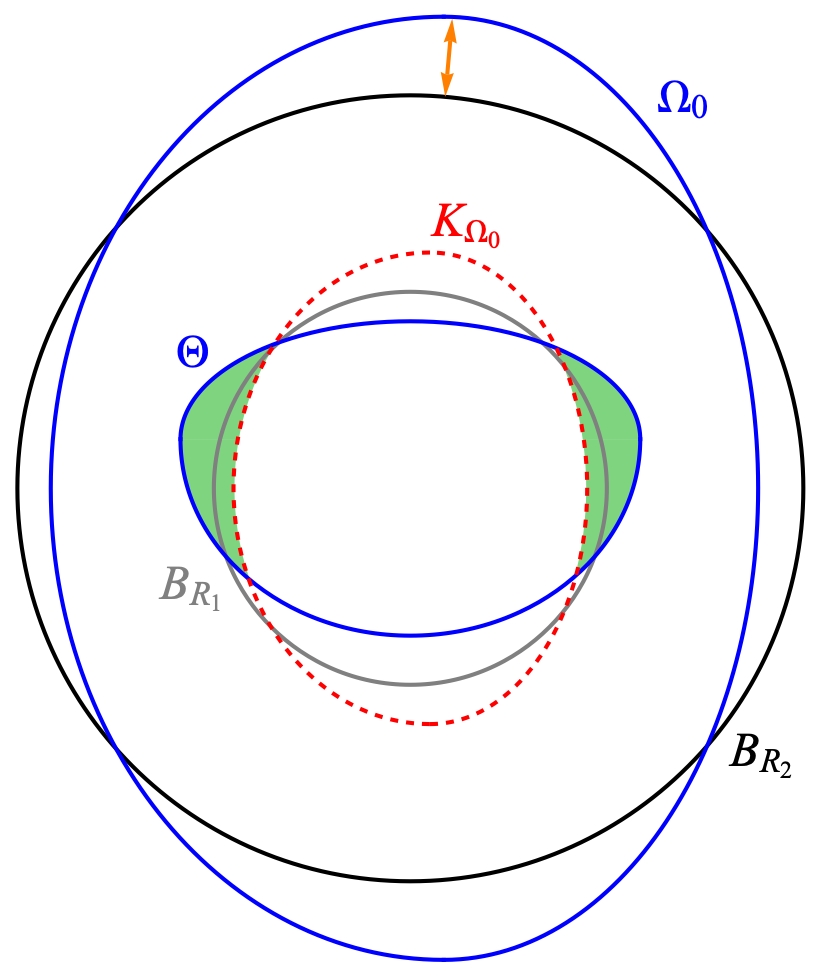}
\caption{$\mathcal{A}_\H(\Omega_0)$ is drawn in orange; the green set is the integration domain in the definition of $\tilde{\mathcal A}(\Theta;\Omega_0)$.}
    \label{fig_01_hybrid}
\end{figure}

We point out that the hybrid asymmetry just introduced is zero if and only if $\Omega=A_{R_1,R_2}$, and so, \eqref{hybrid} actually  quantifies a deviation of an admissible set from the optimal set.
%The quantities $g(\mathcal{A}_\H(\Omega_0))$ and $\tilde{\mathcal{A}}(\Theta;\Omega_0)$  for a proper definition of an asymmetry functional in our setting. 
%Both asymmetry indexes are necessary to estimate the deviation of $\Omega_0\setminus\overline{\Theta}$ from $A_{R_1,R_2}$: 
The term $g(\mathcal{A}_\H(\Omega_0))$ quantifies the deviation with respect to the Hausdorff distance of the domain $\Omega_0$ from the optimal exterior shape $B_{R_2}$; $\tilde{\mathcal{A}}(\Theta;\Omega_0)$ is necessary to estimate the total  deviation of $\Omega$ from $A_{R_1,R_2}$.
%henever $\Omega_0$ has very low (in fact, zero) Hausdorff asymmetry and it is a sort of "semi-distance" of the %an integral distance
%of the inner hole $\Theta$ from the hole $K_{\Omega_0}$, which has the "closest" shape to $\Omega_0$.  
Indeed, if $\Theta$ has a low Hausdorff distance from $K_{\Omega_0}$, then the larger term in the hybrid asymmetry is $g(\mathcal{A}_H(\Omega_0))$; on the other hand, if $\Omega_0$ is the ball $B_{R_2}$, any admissible hole $\Theta\neq B_{R_1}$ satisfies $\tilde{\mathcal{A}}(\Theta;B_{R_2})>0$ and thus in that case the largest term is $\tilde{\mathcal{A}}(\Theta;B_{R_2})$ (see Remark \ref{renna}).     

Moreover, $\alpha(\Omega)$ is a proper asymmetry and in $\mathcal{T}_{R_1,R_2}$ it holds (see Proposition \ref{tendea0}) that 
\[\lim_{j\to+\infty} d_{\mathcal H}(\Omega_j,A_{R_1,R_2}) =0  \iff \lim_{j\to+\infty}\alpha(\Omega_j)=0.\]

The key idea that inspired our definition of asymmetry 
is the following:
provided that $\Omega_0$ is the outer box, the most natural hole to compare $\Theta$ with  is neither a ball with the same measure as $\Theta$ nor $B_{R_1}$, but the inner parallel $K_{\Omega_0}$ having the same measure as $\Theta$. This behaviour seems to depend on the fact that  $K_{\Omega_0}$ makes the set $\Omega_0\setminus K_{\Omega_0}$  a "shell" with outer "box" $\Omega_0$ and fixed distance of the inner boundary $\partial K_{\Omega_0}$ from the outer boundary $\partial\Omega_0$. Moreover, the test function built via web function is constant on both $\partial K_{\Omega_0}$ and $\partial\Omega_0$. %For those reasons, fixed $\Omega_0$, it is more natural to compare any inner hole $\Theta$ with $K_{\Omega_0}$, not with a ball with the same measure of $\Theta$. 
Nevertheless, the set $K_{\Omega_0}$ coincides with $B_{R_1}$ if and only if $\Omega_0=B_{R_2}$.

%that $\Omega_0$ has a given Hausdorff asymmetry with respect to the ball, $\Theta$ has a given distance  from $K_{\Omega_0}$. The hybrid asymmetry that we find is zero if and only if $\Omega_0\setminus\Theta=A_{R_1,R_2}$, and so this mixed asymmetry actually quantifies a deviation of an admissible set from the optimal set. 

The main results of the paper are the following.  The first result is the quantitative version of \eqref{ineq_eig}.

\begin{mtheorem1}[Quantitative isoperimetric inequality]\label{teo_1}
Let $\Omega\subset\R^n$ such that $\Omega=\Omega_0\setminus\overline{\Theta}$, with $\Omega_0$ convex and $\Theta$ is a finite union of open sets homeomorphic to balls. Let $A=A_{R_1,R_2}$ the spherical shell such that $P(B_{R_2})=P(\Omega_0)$ and $|\Omega|=|A|$.
There exist a positive constant $C(n,R_1,R_2,\beta)$ such that, if $\Omega_0$ is $R_2$-nearly spherical, then
\begin{equation}
\label{stab_ineq_intro1}
\lambda_1(\beta,A)-\lambda_1(\beta,\Omega) \geq C(n,R_1,R_2,\beta) \alpha(\Omega). 
\end{equation}
\end{mtheorem1}

The second result is the stability of inequality \eqref{ineq_eig}  in the class $\mathcal{T}_{R_1,R_2}$.%is that inequality   \eqref{ineq_eig} is stable in the class $\mathcal{T}_{R_1,R_2}$.

\begin{mtheorem2}[Stability in  $\mathcal{T}_{R_1,R_2}$]\label{teo_2}
Let $0<R_1<R_2<+\infty$. There exist two positive constants $C(n,R_1,R_2,\beta)$ and $\delta_0(n,R_1,R_2,\beta)$ such that,  for every $\Omega\in\mathcal{T}_{R_1,R_2}$, if $\lambda_1%^{RN}
(\beta,A)-\lambda_1%^{RN}
(\beta,\Omega)\leq \delta_0$, then
\begin{equation}
\label{stab_ineq_intro2}
\lambda_1(\beta,A)-\lambda_1(\beta,\Omega) \geq C(n,R_1,R_2,\beta) \alpha(\Omega), 
\end{equation}
where $A=A_{R_1,R_2}$ is the spherical shell in $\mathcal{T}_{R_1,R_2}$ and $\alpha(\Omega)$ is the asymmetry defined in \eqref{hybrid}.  
\end{mtheorem2}

The proofs are based on  two  estimates. Since on $\partial \Omega_0$ a Robin boundary condition holds, we can use a Fuglede type technique on an auxiliary Steklov type problem to obtain the first estimate relative to the perturbation of the outer boundary. 
For this type of approach we refer to  \cite{fuglede1989stability, fusco2009stability} for  the stability  of the isoperimetric inequality and to \cite{ ferone2015conjectured, gavitone2020quantitative, cito2021quantitative} for the quantitative estimates of some spectral functionals. 
Then, once   the outer boundary is fixed, we can estimate the deviation of $\Theta$ from $K_{\Omega_0}$ using the weak asymmetry $\tilde{\mathcal A}(\Theta;\Omega_0)$, defined in    \eqref{eq:atilde} and  inspired by  \cite{brandolini2010upper, amato2024estimates}.  

In conclusion, the main novelty  of the paper is the introduction of an hybrid type of asymmetry related to the fact that we have to deal with an holed domain with two different boundary conditions and subject to both a perimeter and a volume constraint.
More specifically, as far as the asymmetry of the outer boundary, when treating the auxiliary Steklov problem in the radial case, instead of having only one Bessel function as in \cite{ferone2015conjectured, cito2021quantitative}, we have a linear combination of two Bessel functions and so we have to face the related technical difficulties, such as the validity of all the recurrence formulas, the derivation rules, and the estimate from below of some integral quantities linked to the quantitative estimate. 
Moreover, the inner hole with Neumann boundary condition cannot be treated as in \cite{fuglede1989stability}, since a cancellation of the inner boundary integrals occurs due to the boundary condition. So,  we cannot use the standard Fuglede approach and we need to introduce the weighted Fraenkel type term that appears in \eqref{hybrid}. 

\bigskip

The structure of the paper is the following. In Section \ref{notations_sec}, we  give  some historical insights,  we fix the notation, we recall some known results on the Robin-Neumann eigenvalue problem and recall some useful results on the nearly spherical sets.  In Section \ref{Steklov_sec} we study an auxiliary problem of Steklov type; in Section \ref{main_sec}, we provide some quantitative estimates for $\lambda_1(\beta,\Omega)$ and we prove of the main result of the paper. Finally, in Section \ref{rem_sec}, a list of open problems is presented.

\section{Background and Preliminary results}
\label{notations_sec}
We start this Section by describing the historical background on problems with different boundary conditions on the outer and inner boundary. Then, we fix the notation, we recall some  results on the Robin-Neumann eigenvalue problem and some properties of nearly spherical sets.

\subsection{ A glimpse of history}
%In  recent years, the study of eigenvalue problems, when different boundary conditions has been imposed on the outer and inner boundary of the domain, have been the object of much interest. In literature, up to our knowledge, Stanislaw Zaremba in 1910 \cite{zaremba1910probleme} firstly solved a boundary value problem, suggested to him by Wirt 
%with mixed boundary condition (that is the problem of finding a solution of a partial differential equation satisfying a Dirichlet or a Neumann boundary condition on disjoint parts of the boundary) for the Laplace equation. 

The study of spectral problems for the Robin-Laplace operator  with negative parameter $\beta$ on holed domains arises from the fact that the optimal shape for the first eigenvalue, among domains with fixed volume, is  expected to be the spherical shell for $\abs{\beta}$ large enough. More precisely, in 1977, Bareket \cite{bareket1977isoperimetric} conjectured that, among all smooth bounded domains of given volume, the first eigenvalue is maximized when $\Omega$ is a ball and, in \cite{ferone2015conjectured}, the conjecture has been proved for domains close to a ball in the $L^\infty$ topology. Subsequently, in \cite{freitas2015first,kovavrik2017p}, it has been proved that the conjecture is true for small values of $\left|\beta\right|$; on the other hand, for $\left|\beta\right|$ great enough, the first Robin eigenvalue on the spherical shell is greater than the eigenvalue on the disk, among domains of fixed volume.

More recently, many authors have investigated the Laplace eigenvalue problem   with different  boundary conditions on the outer and inner boundary: for instance positive Robin-Neumann \cite{paoli2020sharp,payne1961some}, Neumann-positive Robin \cite{dellapietra2020optimal,hersch1962contribution}, Dirichlet-Neumann \cite{anoop2023domain,anoop2021shape}, Dirichlet-Dirichlet \cite{biswas2021optimization,chorwadwala2022optimal},
Steklov-Dirichlet \cite{ftouhi2022place,
gavitone2021isoperimetric,gavitone2024monotonicity, hong2022first, paoli2021stability,
verma2020eigenvalue}, Steklov-Robin \cite{gavitone2023steklov}.

%This is the starting point of our researches regarding the shape optimization on holed domains, for which the natural optimal candidates are the spherical shells. 
In particular, in this paper we study a stability result for this kind of problems. 

To the best of our knowledge, the quantitative improvement of spectral inequalities firstly appeared in \cite{hansen1994isoperimetric,melas1992stability}; quantitative improvements of the P\'olya-Szeg\"o principle have been proven in \cite{cianchi2008quantitative}. It is worth noticing that, when $\Theta=\emptyset$, the quantitative inequalities for the Laplace eigenvalue have been studied when Dirichlet (\cite{bhattacharya2001some,fusco2009stability} and  \cite{brasco2015faber}), Neumann \cite{brasco2012sharp}, positive Robin \cite{bucur2018quantitative}, negative Robin \cite{amato2024estimates, cito2021quantitative} or Steklov \cite{brasco2012spectral,gavitone2020quantitative} boundary condition holds.

On the other hand, the quantitative versions of the spectral inequalities on holed domain are still in the early stages of development. We mention \cite{paoli2021stability}, where a result in this direction is proved for the first Steklov-Dirichlet Laplacian eigenvalue among the class of holed domains where the inner hole $\Theta$ is a given ball well cointained in $\Omega_0$. The inner ball can be only translated and so the only perturbation of the boundary acts on $\partial\Omega_0$. 

Nevertheless, it seems that no results are available if a perturbation of \emph{both} inner and outer boundary is kept into account. Nevertheless, it seems that no results are available if a perturbation of \emph{both} inner and outer boundary is kept into account. We point out the recent result in \cite{amato2023talenti}, where the authors give a quantitative result of the classical  Talenti comparison, using another type of mixed  asymmetry. 

\subsection{Notations}
Throughout this paper, we denote by $|\cdot|$, $P(\cdot)$ and $\mathcal{H}^k(\cdot)$, the $n-$dimensional Lebesgue measure, the perimeter and the $k-$dimensional Hausdorff measure in $\mathbb{R}^n$, respectively. The unit open  ball in $\mathbb{R}^n$ will be denoted by $B_1$ and  $\omega_n:=|B_1|$ and the unit sphere in $\mathbb{R}^n$ by $\mathbb{S}^{n-1}$. 
More generally, we denote with $B_r(x_0)$ the set $x_0+rB_1$, that is the ball centered at $x_0$ with measure $\omega_nr^n$, and, for any $0<R_1<R_2<+\infty$, we denote by  $A_{R_1,R_2}$ the spherical shell $B_{R_2}\setminus \overline{B_{R_1}}$. 

Throughout this paper, $\Omega=\Omega_0\setminus \overline{\Theta}$ is a subset of $\mathbb{R}^n$, $n\geq 2$, such that $\Omega_0$ is an open, bounded and convex set and $\Theta$ is a convex set such that $\overline{\Theta}\subset \Omega_0$.

The distance of a point  from the boundary $\partial\Omega_0$ is the function 
\begin{equation*}
d(x)= \inf_{y\in \partial \Omega_0} |x-y|, \quad x\in \overline{\Omega_0},
\end{equation*}
moreover, the inradius of $\Omega_0$ is the radius of the largest ball contained in $\Omega_0$, i.e.
\begin{equation*}
\rho({\Omega_0})=\max \{d(x),\; x\in\overline{\Omega_0}\}.
\end{equation*}

We denote by $J_\nu(r)$, $Y_\nu(r)$, $I_\nu(r)$ and $K_\nu(r)$ the Bessel function of first and second kind, the modified Bessel functions of the first and second kind of order $\nu$,  respectively.
We recall the following recurrence and derivation formulas, see \cite{ abramowitz2006handbook, watson1922theory}:
\begin{gather}
\label{Besselx1}
I_{\nu-1}(r)=\dfrac{2\nu}{r}I_\nu(r)+I_{\nu+1}(r),\\
\label{Besselx2}
K_{\nu-1}(r)=-\dfrac{2\nu}{r}K_\nu(r)+K_{\nu+1}(r),
\end{gather}
\begin{gather}
\label{besselJ_rule}
 \left(z^{-\nu} J_{\nu}(r)\right)'=-z^{-\nu} J_{\nu+1}(r),\\
 \label{besselY_rule}
\left(z^{-\nu} Y_{\nu}(r)\right)'=-z^{-\nu} Y_{\nu+1}(r),\\
 \label{besselI_rule}
\left(z^{-\nu} I_{\nu}(r)\right)'=z^{-\nu} I_{\nu+1}(r),\\
 \label{besselK_rule}
\left(z^{-\nu} K_{\nu}(r)\right)'=-z^{-\nu} K_{\nu+1}(r),
\end{gather}
\begin{gather}
\label{BesselIderivative}
I_{\nu}'(r)=\dfrac{1}{2}\left(I_{\nu-1}(r)+I_{\nu+1}(r)\right),\\
\label{BesselKderivative}
K_{\nu}'(r)=-\frac{1}{2}\left(K_{\nu-1}(r)+K_{\nu+1}(r)\right).
\end{gather}

\subsection{The Robin-Neumann  eigenvalue problem}
\label{Robin_subsec}
We collect now the definitions and  the basic properties of the first Robin-Neumann eigenvalue; for  the proofs refer to \cite{paoli2020sharp}.  

We start by  recalling  the weak formulation of problem \eqref{eig_prob}.
\begin{definition}
A real number $\lambda$ is an eigenvalue of \eqref{eig_prob} if and only if there exists a function $w\in H^{1}(\Omega)$, not identically zero, such that 
\begin{equation*}
%\label{weakly}
\int_{\Omega}\nabla \psi\nabla\varphi \;dx+\beta \int_{\partial \Omega_0} \psi\varphi\;d\mathcal{H}^{n-1}=\lambda\int_{\Omega}\psi \varphi \;dx
\end{equation*}
for every $\varphi\in H^{1}(\Omega)$. The function $\psi$ is called the eigenfunction associated to $\lambda$.
\end{definition}

In the following definition, we list the main properties of \eqref{eig_min}.
\begin{proposition}
Let $\beta<0$.
\begin{itemize}
\item There exists a minimizer $w\in H^{1}(\Omega)$ of \eqref{eig_min},  which is a weak solution to \eqref{eig_prob}. 
\item  $\lambda_1%^{RN}
(\beta,\Omega)$ is simple, i.e. all the associated eigenfunctions are scalar multiple of each other and can be taken to be positive. Moreover, $\lambda_1%^{RN}
(\beta,\Omega)$ is negative.
\item The map $\beta\mapsto \lambda_1%^{RN}
(\beta,\Omega)$ is Lipschitz continuous, non-decreasing, concave and surjective onto $(-\infty,0)$.
\item Let $R_1,R_2$ be two nonnegative real numbers such that $R_2> R_1$, and let $z$ be the minimizer of problem \eqref{eig_min} on the spherical shell $A_{R_1,R_2}$. Then,  $z$ is strictly positive and radially symmetric, in the sense that $z(x)=:\Psi(|x|)$, with $\Psi'(r)>0$.
\end{itemize}
\end{proposition}

We recall now explicitly Theorem $3.1$ in \cite{paoli2020sharp}, where it is proved that 
the spherical shell is the optimal shape for $\lambda_1(\beta,\Omega)$ when both the volume and the outer perimeter are fixed. This result holds for all $\beta\in \R\setminus \{0\}$, but for our purposes we state it only in the case of negative parameter.
\begin{theorem}[Theorem $3.1$ in \cite{paoli2020sharp}]
\label{max_lambda}
Let $\beta<0$ and  let  $A=A_{R_1,R_2}$  be the spherical shell such that $|A|=|\Omega|$ and  $P(B_{R_2})=P(\Omega_0)$. Then,
\begin{gather*}
\lambda_1%^{RN}
(\beta, \Omega)\leq \lambda_1%^{RN}
(\beta, A).
\end{gather*}
\end{theorem}

The test functions used in the proof of Theorem \ref{max_lambda} fall in the class of the so-called web-functions,  which
depend only on the distance from the boundary. For more details on the web function we refer to \cite{gazzola1999existence, makai1959principal,polya1960two}, and we point out that also the analysis of functionals only depending on the distance has been widely studied, see e.g. \cite{lemenant}.

Regarding problem \eqref{eig_min}, the authors in \cite{paoli2020sharp} use  the following test function
\begin{equation*}%\label{web_function_PPT}
w_{\Omega_0}(x):=
\begin{cases}
G(d(x))\quad &\text{if} \  d(x)< R_2-R_1\\
z_m &\text{if} \ d(x)\geq R_2-R_1
\end{cases}
\quad x\in \Omega_0,
\end{equation*}
where $z$ is the eigenfunction of \eqref{eig_min} on $A_{R_1,R_2}$ and $G$ is defined as
\begin{equation*}
G^{-1}(t)=\int_{t}^{z_M}\dfrac{1}{\ell(\tau)}\;d\tau,
\end{equation*}
and $\ell(t)=|\nabla z|_{z=t}$, determined for $z_m<t\leq z_M$ with $z_m:=\min_A z$  and $z_M:=\max_A z$. Moreover,  $w_{\Omega_0}$ satisfies the following properties: $w_{\Omega_0}\in H^1(\Omega_0)$ and	\begin{gather*}
|\nabla w_{\Omega_0}|_{u=t}=|\nabla z|_{v=t},\\
(w_{\Omega_0})_m:=\min_{\Omega_0} w_{\Omega_0}\geq z_m,\\
(w_{\Omega_0})_M:=\max_{\Omega_0} w_{\Omega_0} = z_M=G(0).
\end{gather*}
%Therefore the function  \eqref{GFerone} is simply an extension to $\Omega_0$ of the test function \eqref{web_function_PPT}.
 
Moreover, we recall  the expression of the eigenfunction $z$ associated to $\lambda=\lambda_1(\beta, A_{R_1,R_2})$: %is obtained by imposing the boundary conditions to a linear combination of the Bessel function of first and second order:
\begin{equation*}%\label{aut_anello}
z(x)=|x|^{1-\frac n2 }\left(Y_{\frac n2 }(\sqrt\lambda R_1)J_{\frac n2 -1}(\sqrt\lambda |x|)-J_{\frac n2 }(\sqrt\lambda R_1)Y_{\frac n2 -1}(\sqrt\lambda |x|)\right).
\end{equation*}
So, by the radiality of $z$, we have $z(x)=\Psi(|x|)$ and 
\[
 \Psi(r)=r^{1-\frac n2}\delta(r),\quad
 \Psi'(r)=r^{1-\frac n2}\gamma(r),
\]
once having set
\[
\begin{split}
& \delta(r)=Y_{\frac n2 }(\sqrt\lambda R_1)J_{\frac n2 -1}(\sqrt\lambda r)-J_{\frac n2 }(\sqrt\lambda R_1)Y_{\frac n2 -1}(\sqrt\lambda r),\\
& \gamma(r)=-Y_{\frac n2 }(\sqrt\lambda R_1)J_{\frac n2}(\sqrt\lambda r)-J_{\frac n2 }(\sqrt\lambda R_1)Y_{\frac n2 }(\sqrt\lambda r),
\end{split}
\]
and  using \eqref{besselJ_rule}-\eqref{besselY_rule}.

Therefore, substituting $z$ in \eqref{eig_min}, we obtain 
\begin{equation}
\label{eig_spherical}
\lambda_1(\beta, A_{R_1,R_2})=\frac{\ds\int_{R_1}^{R_2}r\gamma^2(r)dr+\beta R_2 \delta^2(R_2) }{\ds\int_{R_1}^{R_2}r\delta^2(r)dr}.
\end{equation}
Finally, we stress that it is possible to write the explicit expression of \eqref{eig_spherical} by using the following relations for Bessel functions (see \cite[eq. (10)-(11) pag. 134]{watson1922theory}):
\[
\begin{split}
& \int r F_\nu (\sqrt\lambda r)G_\nu (\sqrt\lambda r)dr=\frac 14 r^2\left(2 F_\nu (\sqrt\lambda r)G_\nu (\sqrt\lambda r)-F_{\nu-1} (\sqrt\lambda r)G_{\nu+1} (\sqrt\lambda r)-F_{\nu+1} (\sqrt\lambda r)G_{\nu-1} (\sqrt\lambda r)\right),\\
& \int r F^2_\nu (\sqrt\lambda r)dr=\frac 12 r^2\left( F^2_\nu (\sqrt\lambda r)-F_{\nu-1} (\sqrt\lambda r)F_{\nu+1} (\sqrt\lambda r)\right),
\end{split}
\]
where $F_\nu$ and $G_\nu$ are Bessel functions of any kind of order $\nu$.

\subsection{Nearly spherical sets}%\label{nearly_sec}
In this Section, we give some definitions and useful properties related to nearly spherical sets (see \cite{fuglede1989stability,fusco2015quantitative}).
\begin{definition}
\label{NSset}
Let $\Omega_0\subset \mathbb{R}^n$ be an open bounded set with $0\in \Omega_0$. The set  $\Omega_0$ is said a \emph{$R-$nearly spherical set parametrized by $u$} if there exist a constant  $R>0$ and $u\in W^{1,\infty}(\mathbb{S}^{n-1})$, with $||u||_{W^{1,\infty}}\le R/2$, such that
\begin{equation*}%\label{nearly}
 \partial    \Omega_0= \left\{y \in \R^n \colon y=\xi  (R+u(\xi)), \ \xi \in \mathbb{S}^{n-1}\right\}.
\end{equation*}
%where $R$ % = (\omega \slash \omega_n)^{\frac{1}{n}}$ is the radius of the ball having the same measure of $\Omega_0$.
\end{definition}
Let us observe that the volume of a $R-$nearly spherical set is given by 
\begin{equation*}
%\label{vol_ns}
|\Omega_0|=\dfrac{1}{n}\int_{\mathbb{S}^{n-1}} \left(R+u(\xi) \right)^n\;d\mathcal{H}^{n-1},
\end{equation*}
and the perimeter by
\begin{equation*}%\label{perimeter}
P(\Omega_0)=\int_{\mathbb{S}^{n-1}}\left(R+u(\xi)\right)^{n-1} \sqrt{1+\dfrac{|\nabla_\tau u(\xi)|^2}{(R+u(\xi))^2}}d\mathcal{H}^{n-1},
\end{equation*}
where $\nabla_\tau u$ denotes  the tangential gradient of $u$ and the tangential Jacobian of the map  \\$\phi:\xi\in\mathbb{S}^{n-1}\to y=R\xi (1+u(\xi))$ is 
\begin{equation*}%\label{jac}
J_{\phi}(\xi)= (R+u(\xi))^{n-2}\sqrt{(R+u(\xi))^2+|\nabla_\tau u(\xi)|^2}.
\end{equation*}
Moreover, the unit  outer normal to $\partial \Omega_0$ at $y=\xi(R+u(\xi))$ is given by 
\begin{equation*}%\label{normal}
\nu(\xi)=\dfrac{\xi(R+u(\xi))-\nabla_\tau u(\xi)}{\sqrt{(R+u(\xi))^2+|\nabla_\tau u(\xi)|^2}}.
\end{equation*}

In the following, we deal with sets that are, in a suitable sense, close to a spherical shell
 and here we set the needed notation.

%Therefore, we need to introduce a sort of \ \lq\lq doubly connected nearly-spherical set\rq\rq\ and here we set the notation.  % in our setting 
\begin{definition}%\label{def:nearly}
Let $\Omega=\Omega_0\setminus \Theta$, where $\Omega_0$ and $\Theta$ are two convex sets such that $\Theta\subset\subset\Omega_0$ and $0\in \Theta$. We say that $\Omega$ is a \emph{$(R_1,R_2)-$nearly annular set}  if there exist $0<R_1<R_2<+\infty$ and $u,v\in W^{1,\infty}(\mathbb{S}^{n-1})$, with $||u||_{W^{1,\infty}}<R_2/2$, $||v||_{W^{1,\infty}}\le R_1/2$, such that 
\begin{equation*}%\label{nearly_out}
\partial\Omega_0= \left\{y \in \R^n \colon y=\xi  (R_2+u(\xi)), \, \xi \in \mathbb{S}^{n-1}\right\}
\end{equation*}
and
\begin{equation*}%\label{nearly_in}
\partial\Theta= \left\{y \in \R^n \colon y=\xi  (R_1+v(\xi)), \, \xi \in \mathbb{S}^{n-1}\right\}.
\end{equation*}
%where $R_1$ and $R_2$ are such that $|\Omega|=|A_{R_1,R_2}|$ and $P(\Omega_0)=P(B_{R_2})$

%if there exists $x_0\in\mathbb{R}^n$ such that:\begin{equation}d_{\mathcal{H}}\left(B_{R_2}(x_0),\Omega_0\right)\leq \frac{1}{2}\end{equation}and \begin{equation}d_{\mathcal{H}}\left(B_{R_1}(x_0),\Theta_0\right)\leq \frac{1}{2}\end{equation} $\Omega_0$ is a $R_2$ nearly spherical set and $\Theta$ is a $R_1$ nearly spherical set.
\end{definition}

This definition and the outer perimeter and volume constraints yield to the following Lemma, based on the Taylor expansions.

\begin{lemma} 
\label{lemma2}
 There exists $C=C(n)>0$ such that, for any parametrization couple $u, v\in W^{1,\infty}(\mathbb S^{n-1})$ of a $(R_1,R_2)-$nearly annular set and for every $0<\varepsilon<R_1/2$ such that $||u||_{W^{1,\infty}},$ $||v||_{W^{1,\infty}}\leq \varepsilon$, the following estimates hold:
\begin{equation}\label{E1u}
\left| \left(R_2+u \right)^{n-1}-\left(R_2^{n-1}+(n-1)R_2^{n-2} u+ (n-1) (n-2) R_2^{n-3}\frac{u^2}{2} \right)  \right|\leq C\varepsilon R_2^{n-3} u^2;
\end{equation}
\begin{equation}\label{E1v}
\left| \left(R_1+v \right)^{n-1}-\left(R_1^{n-1}+(n-1)R_1^{n-2} v+ (n-1) (n-2) R_1^{n-3}\frac{v^2}{2} \right)  \right|\leq C\varepsilon R_1^{n-3} v^2;
\end{equation}
\begin{equation}\label{E2u}
1+\frac{|\nabla_\tau u|^2}{2R_2^2}-\sqrt{1+\dfrac{|\nabla_\tau u|^2}{(R_2+u)^2}}\leq C\varepsilon\left(u^2+|\nabla_\tau u|^2\right).
\end{equation}
%\begin{equation}\label{E2v} 1+\frac{|\nabla v|^2}{2R_1^2}-\sqrt{1+\dfrac{|\nabla v|^2}{(R_1+v)^2}}\leq C\varepsilon\left(v^2+|\nabla v|^2\right).\end{equation}
Moreover,
\begin{equation}
\label{moreover_1}
\left| R_2^{n-2}\int_{\mathbb{S}^{n-1}}u\;d\mathcal{H}^{n-1} + \dfrac{R^{n-3}_2}{2(n-1)} \int_{\mathbb{S}^{n-1}}|\nabla_\tau u|^2\;d\mathcal{H}^{n-1} +\dfrac{(n-2)}{2} R_2^{n-3}   \int_{\mathbb{S}^{n-1}}u^2\;d\mathcal{H}^{n-1} \right| \leq C\varepsilon,
\end{equation}
and 
\begin{equation}
\begin{split}        \label{moreover_2}
\left| R_1^{n-1}\int_{\mathbb{S}^{n-1}}v\;d\mathcal{H}^{n-1} +\right.& \dfrac{(n-1)}{2}R_1^{n-1} \int_{\mathbb{S}^{n-1}}v^2\;d\mathcal{H}^{n-1}\\
& \left.- R_2^{n-1}   \int_{\mathbb{S}^{n-1}}u\;d\mathcal{H}^{n-1} -\dfrac{(n-1)}{2}R_2^{n-1}\int_{\mathbb{S}^{n-1}}u^2\;d\mathcal{H}^{n-1} \right| \leq C\varepsilon.
\end{split}
\end{equation}
\end{lemma}
\begin{proof}
The proofs of \eqref{E1u}-\eqref{moreover_2} can be immediately adapted from \cite{fuglede1989stability}. Inequalities   \eqref{E1u} and \eqref{E1v} easily follow, respectively, from
\begin{equation}
    \label{sviluppi}
(R_2+u)^{n-1}=\sum_{k=0}^{n-1}\binom{n-1}{k}u^kR_2^{n-1-k}\quad\text{and}\quad(R_1+v)^{n-1}=\sum_{k=0}^{n-1}\binom{n-1}{k}v^kR_1^{n-1-k}.
\end{equation}
Inequality \eqref{E2u} is a consequence of  inequality $\sqrt{1+x}\geq 1+\frac{x}{2}-\frac{x^2}{4}$.
Furthermore, using  the outer perimeter constraint
\[
\frac 1n \int_{\mathbb S^{n-1}} (R_2+u)^{n-1} \sqrt{1+\dfrac{|\nabla u|^2}{(R_2+u)^2}}d\mathcal H^{n-1}=\omega_nR_2^{n-1},
\]
 the first equality in \eqref{sviluppi} and \eqref{E2u}, we obtain \eqref{moreover_1}. 
 Finally,  inequality \eqref{moreover_2} follows by using the volume constraint
\[
\frac 1n \int_{\mathbb S^{n-1}} (R_2+u)^{n-1}- (R_1+v)^{n-1} d\mathcal H^{n-1}=\omega_n(R_2^{n-1}-R_1^{n-1})
\]
and the second equality in  \eqref{sviluppi}.
\end{proof}

The Poincaré inequality in this setting holds with constant $(n-1)$, but, since we are working with functions parametrizing nearly spherical sets, this constant can be improved to $2n$. For the proof we refer to \cite{fuglede1989stability} and \cite[Lem. 2.4]{ferone2015conjectured}.
\begin{lemma} %\label{poincare_lemma}
There exists a constant $C>0$ such that, for any parametrization $u\in W^{1,\infty}(\mathbb S^{n-1})$ of a $R-$nearly spherical set and for every $0<\varepsilon<R/2$  with $||u||_{W^{1,\infty}}\leq \varepsilon$, it holds
\begin{equation}
\label{poincare}
||\nabla_\tau u||_{L^2}^2\geq 2 n|| u||_{L^2}^2-\varepsilon C.
%||D_\tau u||_{L^2}^2+||D_\tau v||_{L^2}^2\geq 2 n \left(|| u||_{L^2}^2+||v||_{L^2}^2 \right)-\epsilon C
 \end{equation}
\end{lemma}
Finally, these last two  Lemmata provide  key estimates for the norms of functions parametrizing the boundary of a nearly spherical set, see  \cite{fuglede1989stability,fusco2015quantitative}.
\begin{lemma}\label{fugledee}
Let $u\in W^{1,\infty}(\mathbb{S}^{n-1})$ such that $\ds\int_{\mathbb{S}^{n-1}}u \;d\mathcal{H}^{n-1}=0$, then
\begin{eqnarray*}%\label{fuglede}
||u||_{L^\infty}^{n-1} \le 
\begin{cases}
\pi \|\nabla_\tau u\|_{L^2} & n=2\\
4||\nabla_\tau u||^2_{L^{2}} \log \frac{8e ||\nabla_\tau u||_{L^{\infty}}^{n-1}}{||\nabla_\tau u||_{L^{2}}^{2}} & n=3
\\[.2cm]
C(n)  ||\nabla_\tau u||_{L^{2}}^{2} ||\nabla_\tau u||_{L^{\infty}}^{n-3}\,\, & n\ge 4.
\end{cases}
\end{eqnarray*}
\end{lemma}
\begin{lemma}\label{utile}
Let $u\in W^{1,\infty}(\mathbb S^{n-1})$ be a parametrization of a $R-$nearly spherical set such that $||u||_{W^{1,\infty}}<\varepsilon$ for some $\varepsilon< R/2$,
then
\[
||\nabla_\tau u||_{\infty}\leq 3 ||u||_\infty^{\frac 12}.
\]
\end{lemma}
\subsection{Some properties of the hybrid asymmetry}
In this Section we focus on some properties of the hybrid asymmetry $\alpha(\Omega)$, see Definition \ref{inner_asymmetry_def}. First of all, we show that $\alpha(\Omega)=0$ if and only if $\Omega=A_{R_1,R_2}$.
 
\begin{remark}
\label{renna}
We point out that  $\tilde{\mathcal A}(\Theta;\Omega_0)\geq 0$, but it could be zero even if $\Omega_0\neq B_{R_2}$ and $\Theta\neq K_{\Omega_0}$. Indeed, if 
$$d_\mathcal{H}(\Theta,K_{\Omega_0})<t_{\Omega_0}-(R_2-R_1),$$
then $\tilde{\mathcal A}(\Theta;\Omega_0)=0$. Nevertheless, in that case, $\Omega_0\neq B_{R_2}$, thus the global asymmetry  satisfies 
\begin{equation*}
   \alpha(\Omega)=\max\left\{g(\mathcal{A}_\H(\Omega_0)),\tilde{\mathcal{A}}(\Theta;\Omega_0)\right\}= g(\mathcal{A}_\H(\Omega_0))>0.
\end{equation*}
 In other words, if $\Theta$ has a low Hausdorff distance from $K_{\Omega_0}$, then the larger term in the computation of $\alpha(\Omega)$ is $g(\mathcal{A}(\Omega_0))$. In particular we point out that, provided that $\Theta\neq K_{\Omega_0}$, $\tilde{\mathcal A}(\Theta;\Omega_0)$ could be zero if and only if $\Omega_0\neq B_{R_2}$. 
 
 On the other hand, we stress the fact  that $t_{\Omega_0}=R_2-R_1$ if and only if $\Omega_0=B_{R_2}$. %, thus the inner weak-asymmetry term is surely strictly positive in the outer worst case ($\Omega_0=B_{R_2}$).
 So, in this case, $K_{\Omega_0}=B_{R_1}$ and any admissible hole $\Theta\neq B_{R_1}$ satisfies $\tilde{\mathcal{A}}(\Theta;B_{R_2})>0$, since $w_{\Omega_0}=z\gneq z_m$ on $\Theta\setminus B_{R_1}$. Consequently, we have
 \begin{equation*}
   \alpha(\Omega)=\max\left\{g(\mathcal{A}_\H(B_{R_2})),\tilde{\mathcal{A}}(\Theta;B_{R_2})\right\}= \tilde{\mathcal{A}}(\Theta;B_{R_2})>0.
\end{equation*}
So, it actually holds $\alpha(\Omega)=0$ if and only if $\Omega=A_{R_1,R_2}$.
\end{remark}

Now, we justify the fact that $\alpha(\Omega)$ is actually an asymmetry, namely that the set $\Omega$ tends to the spherical shell $A_{R_1,R_2}$ in some suitable sense if and only if the hybrid asymmetry $\alpha(\Omega)$ tends to zero.

\begin{proposition}\label{tendea0}
Let $0<R_1<R_2<+\infty$ and $\Omega_j\subset \R^n$ in the class $\mathcal{T}_{R_1,R_2}$. Then 
\[
\lim_{j\to+\infty} d_{\mathcal H}(\Omega_j,A_{R_1,R_2}) =0  \iff \lim_{j\to+\infty}\alpha(\Omega_j)=0.
\]
\end{proposition}
\begin{proof}
The necessary implication directly follows from Definition \ref{inner_asymmetry_def}. Therefore, it only remains to prove the converse condition.

Let $\{\Omega_j\}_{j\in\N}$ be a sequence such that $\alpha(\Omega_j)\to 0$.
Since $\Omega_j\in \mathcal T_{R_1,R_2}$ for any $j\in\N$, then $\Omega_j=(\Omega_0)_j\setminus \Theta_j$. Consequently, from Definition \ref{inner_asymmetry_def}, it is clear that $d_{\mathcal H}({\Omega_0}_j, B_{R_2})\to 0$.

By contradiction, let us assume that $d_{\mathcal H}(\Theta j, B_{R_1})>c$.  We have two possibilities: the difference between $\Theta$ and $B_{R_1}$ is more concentrated either inside or outside $B_{R_1}$. Specifically
\begin{enumerate}
\item[(i)] there exists $x_j\in \Theta_j \cap B_{R_1}$ such that $d(x_j,B_{R_1})>c$ and $d_{\mathcal H}(\partial \Theta_j\setminus B_{R_1},B_{R_1})\to 0$;
\item[(ii)] there exists $x_j\in \Theta_j \setminus B_{R_1}$ such that $d(x_j,B_{R_1})>c$.
\end{enumerate}

In the case (i), the volume constraint gives the contradiction. Indeed, there exists $\bar j$ such that $|\Theta_j|<|B_{R_1}|-c$, for any $j>\bar j$. From the volume constraint, we have an absurd, since it would be $|\Omega_j|\ge|A_{R_1,R_2}|+c$.

Now, we consider the case (ii), where a key role is played by the set $({\Omega_0}_j)_{R_2-R_1}$, in which, by the definition \eqref{GFerone}, the test function $w_{\Omega_0}$ is constant. 

We point out that, in view of the behavior of the outer asymmetry, we have $d_{\mathcal H}(B_{R_1}, ({\Omega_0}_j)_{R_2-R_1})\to 0$; in particular, there exists $c>0$ such that $d_{\mathcal H}(({\Omega_0}_j)_{R_2-R_1},B_{R_1})<c/2$. 
This inequality, together with the assumption (ii), implies that there exists two positive constants $c$ and $c'$ such that $d_{\mathcal H}(({\Omega_0}_j)_{R_2-R_1},\Theta_j)>c/2$ and, in view of the inradius constraint in the class $\mathcal{T}_{R_1,R_2}$, it holds $|\Theta_j\setminus ({\Omega_0}_j)_{R_2-R_1}|>c'/2$. 

Since, by construction and by the volume constraint, we have that $K_{(\Omega_0)_j}\subseteq ({\Omega_0}_j)_{R_2-R_1}$, then
$d(x)< R_2-R_1-c/2$ and hence $w_{(\Omega_0)_j}(x)>G(R_2-R_1-c/2)$ on $\Theta_j\setminus K_{(\Omega_0)_j}$ (and $|\nabla w_{(\Omega_0)_j}|
$ is uniformly bounded from below by a positive constant as well on $\Theta_j\setminus K_{(\Omega_0)_j}$). Consequently, there exists $c''>0$ such that 
$$\tilde{\mathcal{A}}(\Theta_j,{\Omega_0}_j)=\int_{\Theta_j\setminus K_{(\Omega_0)_j}}(|\nabla w_{(\Omega_0)_j}^2|^2+w_{(\Omega_0)_j}^2-z_m^2)\  dx>c''>0,$$
that is absurd.

Therefore, in both cases we contradict the assumption $d_{\mathcal H}(\Theta j, B_{R_1})>c$. Hence we conclude.
\end{proof}

More generally, one can observe that
\[
\lim_{j\to+\infty}\alpha(\Omega_j)=0\iff \lim_{j\to+\infty} d_{\mathcal H}({\Omega_0}_j,B_{R_2}) =0\quad \text{and}\quad \lim_{j\to+\infty}|\Theta_j\Delta B_{R_1}|=0.
\]

Let us observe that, if we denote by $\chi_{E}$ the characteristic function of a set $E$, then the weak weighted Fraenkel asymmetry \eqref{eq:atilde} can be also written as
\[
\tilde{\mathcal{A}}(\Theta,{\Omega_0})=\int_{\R^n}\chi_{\Theta\setminus K_{\Omega_0}}(|\nabla w_{(\Omega_0)}|^2+w_{\Omega_0}^2-z_m^2)dx=Vol_{(|\nabla w_{(\Omega_0)}|^2+w_{\Omega_0}^2-z_m^2)}(\Theta\Delta K_{\Omega_0}).
\]

Finally, it is worth noticing that \[
\tilde{\mathcal{A}}(\Theta,{\Omega_0})\geq 2 z_m ||w_{\Omega_0}-z_m||_{L^1(\Theta\setminus K_{\Omega_0})}.
\]

\section{A quantitative result for a Steklov-type problem}
\label{Steklov_sec}
The quantitative results for the Robin eigenvalue problem with negative boundary parameter are often solved through Steklov type problems (see \cite{bandle2015second,
cito2021quantitative,ferone2015conjectured}). 
This strategy is suggested by a similar behavior of the functionals in terms of monotonicity and of eigenfunctions. Indeed, as a first trivial link between the two functionals, one can observe that the Rayleigh quotients of both the Robin and the Steklov first  eigenvalues are monotonically decreasing with respect to the boundary integral.

The strategy of the proof is inspired by \cite{cito2021quantitative,ferone2015conjectured}, where the authors relate the Raylegh quotient of a suitable Steklov-type eigenvalue problem with  a ratio involving only the boundary integral terms. Subsequently, the suitable representations of the outer and inner boundary  are useful to estimate the difference between the eigenvalues on the radial and on the quasi-radial shape.

\subsection{A Steklov-type problem}%\label{sub_sec_Steklov_pb} 
The auxiliary
problem we deal with is the following:
\begin{equation}\label{eig_pezzotto}	
\sigma_1( \Omega)=  \min_{\substack{\psi\in H^{1}(\Omega)\\ \psi\not \equiv0}} \dfrac{\ds\int _{\Omega}|\nabla\psi|^2+\int _{\Omega}\psi^2\;dx}{\ds\int_{\partial\Omega_0}\psi^2\;d\mathcal H^{n-1}}.
\end{equation}
If $w\in H^1(\Omega)$ is a minimizer of \eqref{pezzotto_sistema}, then it satisfies:
\begin{equation}
\label{pezzotto_sistema}
\begin{cases}
-\Delta w+w=0 & \mbox{in}\ \Omega \vspace{0.2cm}\\
\dfrac{\partial w}{\partial \nu} =\sigma_1(\Omega) w  &\mbox{on}\ \partial \Omega_0 \vspace{0.2cm}\\ 
\dfrac{\partial w}{\partial \nu} =0 &\text{on }\partial \Theta.
\end{cases}
\end{equation}
If  $\Omega$ is the spherical shell $A_{R_1,R_2}=B_{R_2}\setminus\overline{B}_{R_1}$, the unique solution (up to multiplicative constant) of problem \eqref{eig_pezzotto} is radial and it is given by
\[
\mathcal{Z}(x)=\Psi(|x|),
\]
where
\begin{equation}
\label{exact_sol_annulus}
\Psi(r)=r^{1-\frac{n}{2}} K_{\frac{n}{2}}(R_1) I_{\frac{n}{2}-1}(r)+  r^{1-\frac{n}{2}} I_{\frac{n}{2}}(R_1) K_{\frac{n}{2}-1}(r)\quad \forall r\in[R_1,R_2].
\end{equation}

The starting point of our analysis is  the following upper bound for $\sigma_1(\Omega)$, only involving outer boundary integral terms.
\begin{lemma}
\label{lemma1}
It holds that
\begin{equation}
\label{quotient}
\sigma_1(\Omega)\leq \dfrac{\displaystyle\int_{\partial \Omega_0} \mathcal{Z} \dfrac{\partial \mathcal{Z}}{\partial \nu} \;d\mathcal{H}^{n-1}}{\displaystyle\int_{\partial \Omega_0} \mathcal{Z}^2\;d\mathcal{H}^{n-1}}=: \dfrac{N(\Omega)}{D(\Omega)},
\end{equation}
where $\mathcal{Z}(x)=\Psi(|x|)$ is defined in \eqref{exact_sol_annulus}. Moreover, equality  in \eqref{quotient} holds iff $\Omega=A_{R_1,R_2}$, that is the spherical shell with the same volume as $\Omega$ and such that $P(\Omega_0)=P(B_{R_2})$.
\end{lemma}

\begin{proof}
Let $w\in H^{1}(\Omega)$ be a solution of \eqref{eig_pezzotto}. For every $\varphi\in H^{1}(\Omega)$, we have that
\begin{align*}
0=\displaystyle\int_{\Omega} \varphi w \;dx -\displaystyle\int_{\Omega} \varphi \Delta w\;dx =\displaystyle\int_{\Omega} \varphi w \;dx+\displaystyle\int_\Omega  \nabla \varphi \nabla w \;dx- \displaystyle\int_{\partial \Omega_0} \varphi \dfrac{\partial w}{\partial \nu}\;d\mathcal{H}^{n-1}=\\
=\displaystyle\int_{\Omega} \varphi w \;dx+\displaystyle\int_\Omega  \nabla \varphi \nabla w \;dx-\sigma_1(\Omega)\displaystyle\int_{\partial \Omega_0} \varphi w \;d\mathcal{H}^{n-1}.
\end{align*}
Therefore, we have  
\begin{equation*}
\sigma_1(\Omega)\displaystyle\int_{\partial \Omega_0} \varphi w \;d\mathcal{H}^{n-1}= \displaystyle\int_{\Omega} \varphi w \;dx+\displaystyle\int_\Omega  \nabla \varphi \nabla w \;dx.
\end{equation*}
If we choose $\varphi=w$, the divergence theorem and the fact that $w$ is a solution of \eqref{pezzotto_sistema} lead to
\begin{align*}
\sigma_1(\Omega)&= \dfrac{\displaystyle\int_{\Omega} w^2 \;dx+\displaystyle\int_\Omega  |\nabla w|^2 \;dx }{\displaystyle\int_{\partial \Omega_0} w^2 \;d\mathcal{H}^{n-1}}= \\ 
&=\dfrac{\displaystyle\int_{\Omega} w^2 \;dx -\displaystyle\int_{\Omega} w \Delta w \;dx+\displaystyle\int_{\partial \Omega_0} w\dfrac{\partial w}{\partial \nu}\;d\mathcal{H}^{n-1} }{\displaystyle\int_{\partial \Omega_0} w^2 \;d\mathcal{H}^{n-1}}=\dfrac{  \displaystyle\int_{\partial \Omega_0} w\dfrac{\partial w}{\partial \nu}\;d\mathcal{H}^{n-1} }{\displaystyle\int_{\partial \Omega_0} w^2 \;d\mathcal{H}^{n-1}}.
\end{align*}
In this way, we have obtained the variational characterization of $\sigma_1(\Omega)$:
\begin{equation}
\label{variational_steklov}
\sigma_1(\Omega)= \min_{\substack{\psi\in H^{1}(\Omega)\\\psi\not \equiv0}} \dfrac{  \displaystyle\int_{\partial \Omega_0} \psi\dfrac{\partial \psi}{\partial \nu}\;d\mathcal{H}^{n-1} }{\displaystyle\int_{\partial \Omega_0} \psi^2 \;d\mathcal{H}^{n-1}}.
\end{equation}
Finally, let $\mathcal{Z}$ be the solution of \eqref{eig_pezzotto} when $\Omega$ is the spherical shell; the claimed result follows by testing  the problem \eqref{variational_steklov} with $z$.
%Since we have that $\Delta z + z=0$, multiplying by $z$, integrating, using the divergence theorem and the boundary conditions in \eqref{pezzotto_sistema}, we obtain
%\begin{align}
%  0=-\int_{\Omega} \Delta  u u  \;dx+\int_{\Omega} u^2 \;dx= \int_{\Omega} |\nabla u|^2\:dx-\int_{\partial \Omega_0} z \dfrac{\partial z }{\partial \nu}\;\d\mathcal{H}^{n-1}+\int_{\Omega} z^2 \;dx
 %   \end{align}
\end{proof}

\subsection{The weighted quotient}
Now, our aim is  to write the ratio in the right-hand side of \eqref{quotient} in the case when $\Omega_0$ is a $R_2$-nearly spherical set. It is worth noticing that, since in \eqref{quotient} an interior boundary term is not  appearing, the same variational characterization holds when $\Omega$ is a $(R_1,R_2)-$nearly annular set. 

%For the rest of this Section, to improve the readability, we set  $$\nu=\frac{n}{2}-1$$. Hence, we write t
Keeping in mind the radial solution \eqref{exact_sol_annulus},
%\[\Psi(r)= r^{-\nu}[K_{\nu+1}(R_1) I_{\nu}(r)+ I_{\nu+1}(R_1) K_{\nu}(r)]\quad \forall r\in[R_1,R_2].\]
we define 
\begin{equation}
\label{eta_}
\eta(r)=\left[K_{\frac n2}(R_1) I_{\frac n2-1}(r)+ I_{\frac n2}(R_1) K_{\frac n2-1}(r)  \right] \quad \forall r\in[R_1,R_2],
\end{equation}
\begin{equation}
\label{mu_}
\mu(r)=\left[    K_{\frac n2}(R_1) I_{\frac n2}(r)- I_{\frac n2}(R_1) K_{\frac n2}(r)\right]\quad \forall r\in[R_1,R_2],
\end{equation}
where the second relation is obtained by deriving the first one with the use of the derivative rules \eqref{besselI_rule} and \eqref{besselK_rule}.

We observe that both the quantities $\eta(r)$ and $\mu(r)$ are strictly positive. Indeed,  $\eta>0$  is trivially verified and, as far as $\mu$ is concerned, for any $\nu>-1$, we have that
\begin{equation*}
\dfrac{I_{\nu+1}}{K_{\nu+1}}(R_1)<\dfrac{I_{\nu+1}}{K_{\nu+1}}(r)
\end{equation*}
for any $r>R_1$, since
\begin{equation*}
\frac{d}{dt}\left[  \dfrac{I_{\nu}}{K_{\nu}}(r) \right]=\dfrac{\left(I_{\nu-1}(r)+I_{\nu +1}(r)\right) K_{\nu}(r)+I_{\nu}(r)\left( K_{\nu-1}(r)+K_{\nu+1}(r)\right)}{2 K_\nu(x)^2}>0.
\end{equation*}
Furthermore, the monotonicity of $I_\nu$ and $K_\nu$ with respect to the index (see as a reference \cite{cochran1967monotonicity})  implies that $\eta(r)>\mu(r)$ .
Let us observe that, by \eqref{besselI_rule} and \eqref{besselK_rule}, the following derivation rules also hold:
\begin{equation}
\label{mu_eta_derivative_rules}
\left(r^{1-\frac n2} \eta(r)\right)'= r^{1-\frac n2}\mu(r);\quad \left(r^{-\frac n2} \mu(r)\right)'= r^{-\frac n2} \eta(r);
\end{equation}
\begin{equation}
\label{mu_eta_derivative_rules2}
\eta'(r)=\left(\frac n2-1\right)\frac { \eta(r)}r+\mu(r);\qquad\mu'(r)=-\frac n2 \frac{\mu(r)}r+\eta(r).
\end{equation}
Hence, the functions $\eta$ and $\mu$, defined in \eqref{eta_} and \eqref{mu_} respectively, lead to the definitions of the integrand functions in the quotient \eqref{quotient}, that are
\begin{equation}
\label{h}
    h(r)=r^{2-n} \eta^2(r)\quad \forall r\in[R_1,R_2],
\end{equation}
\begin{equation}
\label{f}
    f(r)=\frac{h'(r)}{2}=r^{2-n} \eta(r)\mu(r)\quad \forall r\in[R_1,R_2].
\end{equation}
%\begin{equation}\label{h}h(r)=\left[ r^{-\nu}K_{\nu+1}(R_1) I_{\nu}(r)+ r^{-\nu} I_{\nu+1}(R_1) K_{\nu}(r)  \right]^2.\end{equation}
%\begin{equation}\label{f}    f(r)=\dfrac{h'(r)}{2}=  r^{-2\nu}\left[ K_{\nu+1}(R_1) I_{\nu}(r)+  I_{\nu+1}(R_1) K_{\nu}(r)  \right] \left[K_{\nu+1}(R_1) I_{\nu+1}(r)- I_{\nu+1}(R_1) K_{\nu+1}(r)   \right ],\end{equation}

Finally, using the parametrization of $R_2-$nearly spherical sets (see Definition \eqref{NSset}), we get:
\begin{equation}
\label{numerator}    N(\Omega)=\displaystyle\int_{\mathbb{S}^{n-1}} f(R_2+u)\left(R_2+u\right)^{n-1}\;d\mathcal{H}^{n-1},
\end{equation}
\begin{equation}
\label{denominator}    D(\Omega)=\displaystyle\int_{\mathbb{S}^{n-1}} h(R_2+u)\left(R_2+u\right)^{n-1} \sqrt{1+\dfrac{|\nabla_\tau u|^2}{\left(R_2+u\right)^2}}\;d\mathcal{H}^{n-1}.
\end{equation}
Let us notice that, in the case $\Omega=A_{R_1,R_2}$, we have
\begin{equation*}
    \sigma_1(A_{R_1,R_2})= \dfrac{f(R_2)}{h(R_2)}.
\end{equation*}

The following Lemma is a consequence of the fact that 
 the functions $h$ and $f$ defined in \eqref{h} and \eqref{f} are analytic.

\begin{lemma}
\label{lemma_sviluppi_hf}
For any parametrization couple $u, v\in W^{1,\infty}(\mathbb S^{n-1})$ of a $(R_1,R_2)-$nearly annular set and for every $0<\varepsilon<R_1/2$ such that $||u||_{W^{1,\infty}},||v||_{W^{1,\infty}}\leq \varepsilon$, there exists $C=C(n,R_1,R_2)>0$ such that, on $\mathbb{S}^{n-1}$, we have:
\begin{equation}
\label{T1u}
\left| h(R_2+u)-h(R_2)-h'(R_2)u-h''(R_2)\frac{u^2}{2}\right|\leq K\varepsilon u^2,
\end{equation}
\begin{equation}
\label{T1v}
\left| h(R_1+v)-h(R_1)-h'(R_1)v-h''(R_1)\frac{v^2}{2}\right|\leq K\varepsilon v^2,
\end{equation}
\begin{equation}
\label{T2u}
\left| f(R_2+u)-f(R_2)-f'(R_2)u-f''(R_2)\frac{u^2}{2}\right|\leq K\varepsilon u^2.
\end{equation}
\end{lemma}
The following monotonicity result will be useful in proving the stability result for $\sigma_1(\Omega)$. The proof generalizes a result contained in \cite{amos1974computation}.
\begin{lemma}\label{lemma_amos}
For any $R_1>0$, we have that
\[
\left(\frac{\eta(r)}{\mu (r)}\right)'<0 \quad\forall\  r\geq R_1,
%\frac{d}{dr}\left[\frac{K_{\nu+1}(R_1)I_{\nu}(r)+I_{\nu+1}(R_1)K_\nu(r)}{K_{\nu+1}(R_1)I_{\nu+1}(r)-I_{\nu+1}(R_1)K_{\nu+1}(r)}\right]<0\quad\forall\  r\geq R_1.
\]
where $\eta$ and $\mu$ are defined in \eqref{eta_} and \eqref{mu_}.
\end{lemma}
\begin{proof}
Since $\eta(r)$ is a linear combination of the modified Bessel functions $I_\nu$ and $K_\nu$, it satisfies the following equation:
\[
\eta''+\frac 1 r \eta' -\frac{r^2+(\frac n2-1)^2}{r^2}\eta=0,
\]
that is 
\[
r^2\eta'\eta''+r\eta'^2-\left(r^2+\left(\frac n2-1\right)^2\right)\eta\eta'=0,
\]
i.e.
\[
r\eta'(r\eta')'-\left(r^2+\left(\frac n2-1\right)^2\right)\eta\eta'=0.
\]
Therefore, dividing by $r\eta'$, we obtain:
\begin{equation}
    \label{rmu'}
(r\eta')'=\frac{r^2+(\frac n2-1)^2}r \eta.
\end{equation}
In the same way, we also gain:
\begin{equation}
    \label{reta'}
(r\mu')'=\frac{r^2+(\frac n2)^2}r \mu .
\end{equation}
%The integration by parts and the division by $r^2\eta^2$, leads to\[\frac{\displaystyle\int_0^r r\eta'(r\eta')'dr}{r^2\eta^2}=\frac{\displaystyle\int_0^r(r^2+\eta^2)\eta\eta'dr}{r^2\eta^2}\]and hence\[\frac{\eta'^2}{\eta^2}=1+\frac{\nu^2}{r^2}-\frac{2}{r^2\eta^2}\displaystyle\int_0^rr\eta^2dr.\]
%By the definition of $Y_\nu$, this relation can be written as\[Y_\nu^2=1+\frac{\nu^2}{r^2}-\frac{2}{r^2\eta^2} \int_0^r r\eta^2 dx.\]
%Hence, we have\begin{equation}\label{Yleq}Y_\nu\leq\sqrt{ 1+\frac{\nu^2}{r^2}}.\end{equation}
%By using \eqref{YT} and \eqref{Yleq}, we reach the following inequality\[T_\nu+\frac{\nu}{r}\leq\sqrt{ 1+\frac{\nu^2}{r^2}}.\]
By using \eqref{rmu'} and \eqref{reta'}, we are able to write the following equality
\[
\int_0^r [\eta(s\mu')'-\mu(s\eta')']ds=(n-1)\int_0^r\frac{\eta(s)\mu(s)}{s}ds>0.
\]
By applying the Green Theorem on the l.h.s, we have
\[
s[\eta\mu'-\mu\eta']_0^r=(n-1)\int_0^r\frac{\eta(s)\mu(s)}{s}ds>0,
\]
that leads to
\begin{equation}
    \label{diffY}
\frac{\mu'(r)}{\mu(r)}-\frac{\eta'(r)}{\eta(r)}=\frac{n-1}{r\eta(r)\mu(r)}\int_0^r\frac{\eta(s)\mu(s)}{s}ds>0\quad\forall r>0.
\end{equation}
By using the relation \eqref{diffY},  the thesis follows by observing that
\[
\left(\dfrac{\eta(r)}{\mu (r)}\right)'= \frac{\eta'(r)\mu(r)-\eta(r)\mu'(r)}{\mu^2(r)}=\frac{\eta'(r)\mu(r)-\eta(r)\mu'(r)}{\mu(r)\eta(r)}\frac{\eta(r)}{\mu(r)}=\left(\frac{\eta'(r)}{\eta(r)}-\frac{\mu'(r)}{\mu(r)}\right)\frac{\eta(r)}{\mu(r)}<0. 
\]
\end{proof}

We end this section with a very technical Lemma establishing the sign of a certain quantity, that will be useful to obtain the stability result. The proof is based on the second part of \cite[Prop. 2.5]{ferone2015conjectured}. For sake of simplicity, we set:
\begin{equation}
\label{A0A1A2}
\begin{split}
A_0=& R_2^{n-1}f(R_2)h(R_2),\\
A_1=&R_2^{n-1}[f(R_2)h'(R_2)-f'(R_2)h(R_2)],\\
A_2=&R_2^{n-1}[ f(R_2)h''(R_2)-f''(R_2)h(R_2)].
\end{split}
\end{equation}
\begin{lemma}\label{A0A1A2_lemma}
For any $R_2>0$, we have that
\begin{equation}
\label{segno_pezzo}
(n-1)R_2A_1+R_2^2A_2+2nA_0>0.
\end{equation}
\end{lemma}
\begin{proof}
If we put in evidence $R_2^2$ in \eqref{segno_pezzo}, we obtain $R_2^2[
\frac{(n-1)}{R_2}A_1+A_2+\frac{2n}{R_2^2}A_0]>0$ and, so, we have to prove that 
\[
\left[
\frac{(n-1)}{R_2}A_1+A_2+\frac{2n}{R_2^2}A_0\right]=[r^{n-1}(f(r)h'(r)-h(r)f'(r)]'+2nr^{n-3}h(r)f(r)>0
\]
for $r=R_2$. By using the definitions of $h$ and $f$ in \eqref{h}-\eqref{f} and the derivation rules \eqref{mu_eta_derivative_rules}, the last relation becomes
\[
[r^{2-n}(r\eta^2(r)\mu^2(r)-r\eta^4(r)+(n-1)\eta^3(r)\mu(r))]'+2nr^{1-n}\eta^3(r)\mu(r)>0
\]
and hence
\[
r^{1-n}\eta(r)\mu(r)[2r^2\mu^2(r)+2(n-1)r\eta(r)\mu(r)+(n+1-2r^2)\eta^2(r)]>0.
\]
If we set
$$ \zeta(r):= 2r^2\mu^2(r)+2(n-1)r\eta(r)\mu(r)+(n+1-2r^2)\eta^2(r), $$
we need to show that $\zeta(r)>0$.
We observe that 
\[\zeta'(r)=\frac 1 r \bigl(\zeta(r)+4r\eta(r)\mu(r)+(n^2-2n-3)\eta^2(r)\bigr).
\]
If $n\geq 3$, since $n^2-2n-3\ge 0$ and $\eta(r),\mu(r)>0$ for any $r>0$, we deduce that $r\zeta'(r)\geq \zeta(r)$ and hence that $\zeta(r)>0$ for any $r>0$. Otherwise, if $n=2$, we observe that the function $\xi (r)=r\zeta'(r)$ is increasing for $r>0$ because, by using \eqref{mu_eta_derivative_rules2}, we have that
 \[
\xi'(r)=(r\zeta'(r))'=6r\mu^2(r)+2r\eta^2(r)>0\quad\forall r>0.
\]
Moreover, it is immediate that   $\xi'(0)=\xi(0)=0$. This means that $\xi(r)>0$ and hence $\zeta'(r)>0$ for any $r>0$. Finally, since $\zeta(0)=0$, this implies that $\zeta(r)>0$ for any $r>0$.
\end{proof}

\subsection{The stability result}
\label{stab_stek}
A crucial result to obtain the stability  with respect to the outer boundary is the following stability issue for the functional $\Omega\mapsto\frac{N(\Omega)}{D(\Omega)}$. Up to some necessary technical modification, the proof follows the scheme in \cite{cito2021quantitative,ferone2015conjectured}.

From now on, we say that $\Omega$ is an admissible set if $\Omega=\Omega_0\setminus\overline{\Theta}$ with $\Omega_0$ convex, bounded and open and $\Theta\subset\subset\Omega_0$ a finite union of open sets homeomorphic to balls.

\begin{proposition}
\label{nss} Let $0<R_1<R_2<+\infty$. 
There exists $C_1=C_1(n,R_1,R_2)>0$ such that, for any admissible set $\Omega=\Omega_0\setminus\overline{\Theta}$ with $\Omega_0$ a $R_2$-nearly spherical set and the outer boundary $\partial\Omega_0$  parametrized by $u\in W^{1,\infty}(\mathbb S^{n-1})$, and for every $0<\varepsilon<R_2/2$ such that $||u||_{W^{1,\infty}}\leq \varepsilon$, the following quantitative inequality holds true:
\begin{equation*}%\label{eq:nss}
\frac{N(A_{R_1,R_2})D(\Omega)-D(A_{R_1,R_2})N(\Omega)}{n\omega_n}\ge C_1(n,R_1, R_2)||\nabla_\tau u||_{L^2}^2.
\end{equation*}
\end{proposition}
\begin{proof}
We start by performing a suitable \lq\lq Taylor expansion\rq\rq\ in terms of the deformations of $\partial\Omega_0$ compared to $\partial B_{R_2}$. By using the characterization of $N(\Omega)$ and $D(\Omega)$ in \eqref{numerator} and \eqref{denominator}, we obtain
\begin{equation*}
    \begin{split}      
\frac{N(A_{R_1,R_2})D(\Omega)-D(A_{R_1,R_2})N(\Omega)}{n\omega_n}&=f(R_2)\int_{\mathbb S^{n-1}}h(R_2+u)(R_2+u)^{n-1}\sqrt{1+\frac{|\nabla_\tau u|^2}{(R_2+u)^2}}\;d\mathcal{H}^{n-1}\\
&\quad -h(R_2)\int_{\S^{n-1}}f(R_2+u)(R_2+u)^{n-1}\;d\mathcal{H}^{n-1}.
    \end{split}
\end{equation*}
%By a Taylor expansion up to the second order around $t=0$ of the two integrals
Then, by using the Lemmata \ref{lemma2} and \ref{lemma_sviluppi_hf}, we get
\begin{equation}\label{eq:lungo}
\begin{split}
\frac{N(A_{R_1,R_2})D(\Omega)-D(A_{R_1,R_2})N(\Omega)}{n\omega_n}&\ge \int_{\S^{n-1}}R_2^{2n-2} [f(R_2)h'(R_2)-f'(R_2)h(R_2)]u\;d\mathcal{H}^{n-1}\\
&\quad+\int_{\S^{n-1}}R_2^{2n-3} [2(n-1)(f(R_2)h'(R_2)-f'(R_2)h(R_2))]\frac{u^2}{2}\;d\mathcal{H}^{n-1}\\
&\quad+\int_{\S^{n-1}}R_2^{2n-2} [f(R_2)h''(R_2)-f''(R_2)h(R_2)]\frac{u^2}{2}\;d\mathcal{H}^{n-1}\\
&\quad+\int_{\S^{n-1}}R_2^{2n-4}f(R_2)h(R_2)\frac{|\nabla_\tau u|^2}{2}\;d\mathcal{H}^{n-1}-\varepsilon C_1\|\nabla_\tau u\|^2_{L^2}.\\
%&\ge t^2 \left[\frac{n}{2}(f(R_2)h'(R_2)-f'(R_2)h(R_2))+\frac{f(R_2)h''(R_2)-f''(R_2)h(R_2)}{2}\right]\|u\|^2_{L^2(\S^{n-1})}\\
%&+t^2\left[\frac{f(R_2)h(R_2)}{2}-\frac{f(R_2)h'(R_2)-f'(R_2)h(R_2)}{2(n-1)}\right]\|D_\tau u\|^2_{L^2(\S^{n-1})}
%-\varepsilon C_2t^2\|D_\tau u\|^2_{L^2(\S^{n-1})}
\end{split}
\end{equation}
Now, using the notation in  \eqref{A0A1A2}, inequality \eqref{eq:lungo} becomes
\begin{equation}
    \label{int_trasf}
\begin{split}
\frac{N(A_{R_1,R_2})D(\Omega)-D(A_{R_1,R_2})N(\Omega)}{n\omega_n}&\ge R_2^{n-1} A_1\int_{\S^{n-1}} u\;d\mathcal{H}^{n-1}+R_2^{n-2} [2(n-1)A_1+R_2 A_2]\int_{\S^{n-1}}\frac{u^2}{2}\;d\mathcal{H}^{n-1}\\
&\quad+R_2^{n-3}A_0\int_{\S^{n-1}}\frac{|\nabla_\tau u|^2}{2}\;d\mathcal{H}^{n-1}-\varepsilon C_1\|\nabla_\tau u\|^2_{L^2}.
\end{split}
\end{equation}
Moreover, by the perimeter constraint given by \eqref{moreover_1}, then  \eqref{int_trasf} becomes 

%\begin{equation}\label{fixperi}t\int_{\S^{n-1}} u d\H^{n-1}+t^2\frac{n-2}{2}\int_{\S^{n-1}}u^2 d\H^{n-1} \geq - t^2\frac{1}{2(n-1)}\int_{\S^{n-1}}|D_\tau u|^2 d\H^{n-1}- C_3(n) o(t^2)\end{equation}
\begin{equation}\label{eq:lungo1}
\begin{split}
\frac{N(A_{R_1,R_2})D(\Omega)-D(A_{R_1,R_2})N(\Omega)}{n\omega_n}&\ge R_2^{n-2}\frac{n A_1+R_2A_2}2\int_{\S^{n-1}} u^2\;d\mathcal{H}^{n-1}\\
&+R_2^{n-3}\left(\frac{A_0}2-\frac{R_2A_1}{2(n-1)}\right)\int_{\S^{n-1}}|\nabla_\tau u|^2\;d\mathcal{H}^{n-1}-\varepsilon C_1\|\nabla_\tau u\|^2_{L^2}.
\end{split}
\end{equation}
%\begin{equation}\label{eq:lungo1}\begin{split}&\frac{N(\Omega^\#)D(\Omega)-D(\Omega^\#)N(\Omega)}{n\omega_n}\\&\ge t^2 \left[\frac{n}{2}(f(R_2)h'(R_2)-f'(R_2)h(R_2))+\frac{f(R_2)h''(R_2)-f''(R_2)h(R_2)}{2}\right]\|u\|^2_{L^2(\S^{n-1})}\\&+t^2\left[\frac{f(R_2)h(R_2)}{2}-\frac{f(R_2)h'(R_2)-f'(R_2)h(R_2)}{2(n-1)}\right]\|D_\tau u\|^2_{L^2(\S^{n-1})}-\varepsilon C_2t^2\|D_\tau u\|^2_{L^2(\S^{n-1})}\end{split}\end{equation}
Firstly, let us note that the term in round parenthesis is positive because it is the sum of two positive terms. Indeed, $A_0$, defined in \eqref{A0A1A2}, is positive,
since $h>0$ and $f>0$ (defined respectevely in \eqref{h} and \eqref{f});
%by the positivity of 
%\[\frac{A_0}2-\frac{R_2A_1}{2(n-1)}=\frac{R_2(\eta^2(R_2)-\mu^2(R_2))\eta^2(R_2)}{2(n-1)}+\frac{\eta^3(R_2)\mu(R_2)}{2}>0.\]
%$h%$ \eqref{h} and of $f$ \eqref{f};
meanwhile 
\begin{equation}
    \label{A_1<0}
A_1=R_2^{n-1}[f(R_2)h'(R_2)-f'(R_2)h(R_2)]=f^2(R_2)\left[\frac{h(r)}{f(r)}\right]'_{r=R_2}=f^2(R_2)\left[\frac{\eta(r)}{\mu(r)}\right]'_{r=R_2}<0,
\end{equation}
where the last inequality follows from Lemma \ref{lemma_amos}.

%\textcolor{red}{
%\begin{equation*}
%\left[\frac{f(R_2)h(R_2)}{2}-\frac{f(R_2)h'(R_2)-f'(R_2)h(R_2)}{2(n-1)}\right]=\frac{R_2(\I^2_{\frac{n}{2}-1}(R_2)-\I^2_{\frac{n}{2}}(R_2))}{2(n-1)\I^2_{\frac{n}{2}}(R_2)}>0.
%\end{equation*}
%Sapendo che A1 negativo la conclusione viene subito.}

Therefore, since it is possible to take  $\varepsilon C_1$ arbitrarily small, where the quantity $C_1$ does not depend on $\varepsilon$, the proof is concluded if we verify that
\begin{equation}
\label{eq:alpha}
\frac{n A_1+R_2A_2}2%\left[\frac{n}{2}(f(R_2)h'(R_2)-f'(R_2)h(R_2))+\frac{f(R_2)h''(R_2)-f''(R_2)h(R_2)}{2}\right]
\ge 0.
\end{equation}
 %(the positive constant in the statement is provided by \eqref{eq:alpha}).

Otherwise, if \eqref{eq:alpha} does not hold (i.e. if $\frac{n A_1+R_2A_2}2<0$), we go on with the estimate from below.  More precisely, by using  the Poincaré inequality \eqref{poincare}, from \eqref{eq:lungo1}, we get
\begin{equation*}%\label{eq:lungo2}
\begin{split}
\frac{N(A_{R_1,R_2})D(\Omega)-D(A_{R_1,R_2})N(\Omega)}{n\omega_n}&\ge \frac{R_2^{n-3}}{4n}\left(\frac{n^2-3n}{n-1}R_2A_1+R_2^2A_2+2nA_0\right)\int_{\S^{n-1}} |\nabla_\tau u|^2\;d\mathcal{H}^{n-1}\\
&\quad-\varepsilon C_1\|\nabla_\tau u\|^2_{L^2}.
\end{split}
\end{equation*}
%\begin{equation}\label{eq:lungo2}\begin{split}&\frac{N(\Omega^\#)D(\Omega)-D(\Omega^\#)N(\Omega)}{n\omega_n}\\&\ge \frac{t^2}{2n} \left[\frac{n}{2}(f(R_2)h'(R_2)-f'(R_2)h(R_2))+\frac{f(R_2)h''(R_2)-f''(R_2)h(R_2)}{2}\right]\|D_\tau u\|^2_{L^2(\S^{n-1})}\\&+t^2\left[\frac{f(R_2)h(R_2)}{2}-\frac{f(R_2)h'(R_2)-f'(R_2)h(R_2)}{2(n-1)}\right]\|D_\tau u\|^2_{L^2(\S^{n-1})}-\varepsilon C_2t^2\|D_\tau u\|^2_{L^2(\S^{n-1})}\\&\ge \frac{t^2}{2n} \Bigg[\frac{n^2-3n}{2(n-1)}(f(R_2)h'(R_2)-f'(R_2)h(R_2))+\frac{f(R_2)h''(R_2)-f''(R_2)h(R_2)}{2}\\&+n(f(R_2)h(R_2))\Bigg]\|D_\tau u\|^2_{L^2(\S^{n-1})}-\varepsilon C_2t^2\|D_\tau u\|^2_{L^2(\S^{n-1})}.\end{split}\end{equation}
Finally, in order to conclude the proof,  it remains to  show that 
$$ \left(\frac{n^2-3n}{n-1}R_2A_1+R_2^2A_2+2nA_0\right)>0.$$
This term can be written as the sum of two positive quantities in the following way:
\begin{equation*}
\frac{n^2-3n}{n-1}R_2A_1+R_2^2A_2+2nA_0=\left[(n-1)R_2A_1+R_2^2A_2+2nA_0\right]-\frac{n+1}{n-1}R_2A_1.
\end{equation*}
Indeed, the positive sign of the first addendum of the right hand side of the last equality  is proved in Lemma \ref{A0A1A2_lemma} and the sign of the second one  follows from \eqref{A_1<0}.
%\[C(n,R_2)=(n-1)(f(R_2)h'(R_2)-f'(R_2)h(R_2))+f(R_2)h''(R_2)-f''(R_2)h(R_2)+2n(f(R_2)h(R_2))>0.\]
%\begin{align*}\frac{n^2-3n}{n-1} &(f(R_2)h'(R_2)-f'(R_2)h(R_2))+f(R_2)h''(R_2)-f''(R_2)h(R_2)+2n(f(R_2)h(R_2))\\&=C(n,R_2)-\frac{n+1}{n-1}(f(R_2)h'(R_2)-f'(R_2)h(R_2)),\end{align*}
%$$(f(R_2)h'(R_2)-f'(R_2)h(R_2))\le 0$$ for any $n\ge 2$ and $R_2>0$. 
%Now, in view of the definition of $f$ and $h$, it results that $$f(R_2)h'(R_2)-f'(R_2)h(R_2)=f(R_2)=f^2(R_2)\frac{d}{dt}\left[\frac{h_\rho(t)}{f_\rho(t)}\right]_{t=1}=f^2_\rho(1)\cdot\frac{d}{dt}\left[\frac{\I_{\frac{n}{2}-1}(t\rho)}{\I_{\frac{n}{2}}(t\rho)}\right]_{t=1}.$$
%In Proposition \ref{amos} it has been proved that the last derivative is negative; then \eqref{eq:g1} holds and the proof is concluded.
\end{proof}

We point out that the previous result holds in particular if $\Omega$ is a $(R_1,R_2)-$nearly annular set.

An immediate consequence of  Proposition \ref{nss} is the following stability result for $\sigma_1(\Omega)$. 

\begin{proposition}
Let $0<R_1<R_2<+\infty$. 
There exists $C_1=C_1(n,R_1,R_2)>0$ such that, for any for any admissible set $\Omega=\Omega_0\setminus\overline{\Theta}$ with $\Omega_0$ a $R_2$-nearly spherical set and the outer boundary parametrized by $u\in W^{1,\infty}(\mathbb S^{n-1})$, and for every $0<\varepsilon<R_2/2$ such that $||u||_{W^{1,\infty}}\leq \varepsilon$, the following stability inequality holds true:
%For any $(R_1,R_2)-$nearly annular set $\Omega$ with $P(\Omega_0)=n\omega_n R_2^{n-1}$, having the outer boundary parametrized by $u\in W^{1,\infty}(\mathbb S^{n-1})$, and for every $0<\varepsilon<R_1/2$ such that $||u||_{W^{1,\infty}}\leq \varepsilon$, there exists $C_1=C_1(n,R_2)>0$ such that 
\begin{equation*}
    \sigma_1(A_{R_1,R_2})-\sigma_1(\Omega)\geq  C_1(n,R_1,R_2)||\nabla_\tau u||_{L^2}^2.
\end{equation*}
Moreover, the constant $C_1(n,R_1,R_2)$ depends continuously (actually analytically) on $R_1$ and $R_2$.
\end{proposition}
\begin{proof}
The proof follows by combining Lemma \ref{lemma1} and Proposition \ref{nss}. 
\end{proof}

\section{Proof of the Main Results}
\label{main_sec}
In this section we provide the proof of the main results. First of all, we use the results of Section \ref{stab_stek} related to the auxiliary problem \eqref{eig_pezzotto} to gain the outer stability result. The strategy  follows an idea developed in \cite{cito2021quantitative,ferone2015conjectured}, linking the stability of the Steklov-Neumann eigenvalue \eqref{eig_pezzotto} to the stability of the Robin-Neumann eigenvalue \eqref{eig_min}. 

Furthermore, in the second part of the section, we develop the key idea of considering an additional term to take into account the asymmetry of the inner boundary. 

Finally, in order to complete the proof of Main Theorem 2, we show that we can obtain the stability result in $\mathcal{T}_{R_1,R_2}$ reducing our study to the nearly annular sets.

\subsection{Outer asymmetry}
At this stage, we are able to restrict our study to the class of nearly annular sets. We will use the quantitative estimate given through the Steklov-Neumann eigenvalues in Section \ref{stab_stek}.%, that will be associated to the Robin-Neumann eigenvalues, for suitable parameters.
\begin{proposition}
For any admissible set $\Omega=\Omega_0\setminus\overline{\Theta}$ with $\Omega_0$ a $R_2$-nearly spherical set, with $P(\Omega_0)=n\omega_n R_2^{n-1}$ and  $|\Omega|=|A_{R_1,R_2}|$, having the outer boundary parametrized by $u\in W^{1,\infty}(\mathbb S^{n-1})$, and for every $0<\varepsilon<R_2/2$ such that $||u||_{W^{1,\infty}}\leq \varepsilon$, % if $\lambda_1(\beta,A_{R_1,R_2})-\lambda_1(\beta,\Omega)<\delta$, 
it holds
\begin{equation} \label{robinnearly}
\lambda_1(\beta,A_{R_1,R_2})-\lambda_1(\beta,\Omega) \geq C(n,\beta,R_1,R_2)\|\nabla_\tau u\|_{L^2}^2%+\|D_\tau v\|_{L^2(\S^{n-1})}^2
.
\end{equation}
\end{proposition}
\begin{proof}
Let $\delta>0$ and let $\Omega$ be chosen as in the statement. The map $\beta\mapsto\lambda_1(\beta,\Omega)$ is continuous and monotonically increasing from $(-\infty,0)$ onto $(-\infty,0)$. Then, there exists $\Bar{\beta}\leq\beta$ such that
$$
\lambda_1({\beta},\Omega)=\lambda_1({\Bar{\beta}},A_{R_1,R_2}).
$$
Hence,
$$
 \lambda_1({\beta},A_{R_1,R_2})-\lambda_1({\beta},\Omega)
=\lambda_1({\beta},A_{R_1,R_2})-\lambda_1({\Bar{\beta
}},A_{R_1,R_2}).
$$
Let us consider the positive constant 
$$\kappa:= \sqrt{|\lambda_1({\beta},\Omega)|}=\sqrt{|\lambda_1({\Bar{\beta}},A_{R_1,R_2})|};$$
for the rescaled sets $\kappa\Omega$ and $\kappa A_{R_1,R_2}$ it holds that
\begin{align*}
-\kappa^2&=\lambda_1({\beta},\Omega)\\
&=\kappa^2\lambda_1\left({\frac{\beta}{\kappa}},\kappa\Omega\right)\le\kappa^2\frac{\ds\int_{\kappa\Omega}|\nabla \psi|^2\:dx+\frac{\beta}{\kappa}\int_{\partial(\kappa\Omega_0)}\psi^2\:d\mathcal{H}^{n-1}}{\ds\int_{\kappa\Omega}\psi^2\:dx}\quad\forall \psi  \in H^1(\kappa\Omega)
\end{align*}
and
\begin{align*}
-\kappa^2&=\lambda_1({\Bar{\beta}},A_{R_1,R_2})\\
&=\kappa^2\lambda_1\left({\frac{\Bar{\beta}}{\kappa}},\kappa A_{R_1,R_2}\right)\le\kappa^2\frac{\ds\int_{\kappa A_{R_1,R_2}}|\nabla \Psi|^2\:dx+\frac{\Bar{\beta}}{\kappa}\int_{\partial(\kappa B_{R_2})}\Psi^2\:d\mathcal{H}^{n-1}}{\ds\int_{\kappa A_{R_1,R_2}}\Psi^2\:dx}\quad\forall \Psi  \in H^1(\kappa A_{R_1,R_2}),
\end{align*}
with equality holding if $\psi$ and $\Psi$ are, respectively, the eigenfunctions for the Robin-Neumann problem on $\kappa\Omega$ with parameter $\beta$ and on $\kappa A_{R_1,R_2}$ with boundary parameter $\Bar{\beta}$. Thus we get
%$$\int_{\kappa\Omega}|\nabla w|^2\:dx+\int_{\kappa\Omega}w^2\:dx\ge\frac{\beta}{\kappa}\int_{\partial(\kappa\Omega_0)}w^2\:d\mathcal{H}^{n-1}\quad\forall w  \in H^1(\kappa\Omega)$$
%and
%$$\int_{\kappa\Omega^\#}|\nabla z|^2\:dx+\int_{\kappa\Omega^\#}z^2\:dx\ge\frac{\Bar{\beta}}{\kappa}\int_{\partial(\kappa\Omega_0^\#)}z^2\:d\mathcal{H}^{n-1}\quad\forall z \in H^1(\kappa\Omega^\#),$$
%By dividing both sides by the boundary integrals we get that
$$-\frac{\beta}{\kappa}\le\frac{\ds\int_{\kappa\Omega}|\nabla \psi|^2\:dx+\int_{\kappa\Omega}\psi^2\:dx}{\ds\int_{\partial(\kappa\Omega_0)}\psi^2\:d\mathcal{H}^{n-1}}\quad\forall \psi  \in H^1(\kappa\Omega)$$
and
$$-\frac{\Bar{\beta}}{\kappa}\le\frac{\ds\int_{\kappa A_{R_1,R_2}}|\nabla \Psi|^2\:dx+\int_{\kappa A_{R_1,R_2}}\Psi^2\:dx}{\ds\int_{\partial(\kappa B_{R_2})}\Psi^2\:d\mathcal{H}^{n-1}}\quad\forall \Psi \in H^1(\kappa A_{R_1,R_2}).$$
It follows that the infimum for the Steklov-Neumann problem \eqref{pezzotto_sistema} is achieved on $\kappa\Omega$ and $\kappa A_{R_1,R_2}$ if $\psi$ and $\Psi$ are, respectively, the eigenfunctions for the Robin-Neumann problem on $\kappa\Omega$ with parameter $\beta$ and on $\kappa A_{R_1,R_2}$ with boundary parameter $\Bar{\beta}$. Therefore, we obtain
$$\sigma_1(\kappa\Omega)=-\frac{\beta}{\kappa},\quad\sigma_1(\kappa A_{R_1,R_2})=-\frac{\Bar{\beta}}{\kappa}.$$
If we denote with $z_{\beta}$ the first eigenfunction of $A_{R_1,R_2}$ for the Robin-Neumann problem with parameter $\beta$, using the variational characterization of $\lambda_1(\beta,\Omega)$ and $\lambda_1({\Bar{\beta}},A_{R_1,R_2})$ we have that
\begin{align*}
\lambda_1(\beta,A_{R_1,R_2}) -\lambda_1({\beta},\Omega)&=\lambda_1({\beta},A_{R_1,R_2}) -\lambda_1({\Bar{\beta}},A_{R_1,R_2})\geq ({\beta}-\Bar\beta) \frac{\ds\int_{\partial B_{R_2}} z_\beta ^2\:d\mathcal{H}^{n-1}}{\ds\int_{A_{R_1,R_2}} z_\beta ^2\:dx}\\
&= C_2(n,\beta ,R_1,R_2) (\beta -\Bar{\beta})\ge\kappa C_2(n,\beta,R_1,R_2)(\sigma_1 (\kappa A_{R_1,R_2}) -\sigma_1(\kappa\Omega))\\
&\ge \kappa C_2(n,\beta,R_1,R_2) C_1(n,\kappa^{\frac 1n} R_1,\kappa^{\frac 1n} R_2)||\nabla_\tau u||_{L^2}^2%+||D_\tau v||_{L^2}^2)
,
\end{align*}
where $C_1$ is the constant found in Proposition \ref{nss}. The conclusion follows by setting $C:=\kappa C_2 C_1$.
\end{proof}

Now we are in position to prove the main stability result regarding the outer boundary. 

\begin{theorem}\label{teo:esterno}
Let $0<R_1<R_2<+\infty$. There exist a positive constant $C(n,R_1,R_2,\beta)$ such that, for every admissible set $\Omega=\Omega_0\setminus\overline{\Theta}$ with $\Omega_0$ a $R_2$-nearly spherical set, we have
\begin{equation*}%\label{stab_ineq}
\lambda_1(\beta,A_{R_1,R_2})-\lambda_1(\beta,\Omega) \geq C(n,R_1,R_2,\beta)g(\mathcal{A}_\H(\Omega_0))
\end{equation*}
where $\mathcal{A}_\mathcal{H}$ is the Hausdorff asymmetry defined in \eqref{fraenkel_asy} and $g$ is the positive increasing function defined by
\begin{equation*}
g(s)=\begin{cases}
s^2 &\text{if $n=2$}\\
f^{-1}(s^2) &\text{if $n=3$}\\
s^\frac{n+1}{2} &\text{if $n\ge 4$},
\end{cases}
\end{equation*}
with $f(t)=\sqrt{t \log \left(\frac{1}{t}\right)}$ for $0<t<e^{-1}$. 
\end{theorem}
\begin{proof}  
We have that
$$
\partial\Omega_0=\{\xi(R_2+u(\xi)),\,\,\xi\in \S^{n-1}\}
$$
for some $u\in W^{1,\infty}(\S^{n-1})$ with $\|u\|_{W^{1,\infty}}<R_2/2$. Let $\rho_2$ such that $|B_{\rho_2}|=|\Omega_0|$; then, we have that
\begin{equation} \label{area}
\frac{1}{n} \int_{\S^{n-1}}(R_2+u)^nd\H^{n-1}=\omega_n \rho_2^n.
\end{equation}
Now, let us define the function $h:=(R_2+u)^n-\rho_2^n$.
By a straightforward expansion of the left hand side of \eqref{area}, we get
$$
R_2^n-\rho_2^n=-\frac{1}{n\omega_n}\displaystyle \int_{\S^{n-1}} \sum_{k=1}^n \binom{n}{k} u^kR_2^{n-k} d\H^{n-1},
$$
which immediately implies $|R_2- \rho _2|<C(n)\|u\|_{W^{1,\infty}}$.
Hence, it holds \begin{equation}
    \label{norm_control}
C_2\|h\|_{L^{\infty}}\leq \|u\|_{L^{\infty}} \leq C_3 \|h\|_{L^{\infty}},
\end{equation}where $C_2$ and $C_3$ are constant depending only on the dimension.
Moreover, since
$$\nabla_\tau h = n(R_2 +u)^{n-1} \nabla_\tau u,$$
then
$$
C_4 \|\nabla_\tau u\|_{L^2}\leq \|\nabla_\tau h\|_{L^2} 
\leq C_5 \|\nabla_\tau u \|_{L^2},
$$
where $C_4,C_5$ depend on the dimension and on $R_2$.
Moreover, from \eqref{area}, we know that  $h$ has zero integral, thus we can apply Lemma \ref{fugledee} to $h$ and use \eqref{norm_control} to infer
\begin{eqnarray}\label{fu}
||u||_{L^\infty}^{n-1} \le 
\begin{cases}
\pi \|\nabla_\tau u\|_{L^2} & n=2\\
4||\nabla_\tau u||^2_{L^{2}} \log \frac{8e ||\nabla_\tau u||_{L^{\infty}}^{n-1}}{||\nabla_\tau u||_{L^{2}}^{2}} & n=3
\\[.2cm]
C(n)  ||\nabla_\tau u||_{L^{2}}^{2} ||\nabla_\tau u||_{L^{\infty}}^{n-3}\,\, & n\ge 4.
\end{cases}
\end{eqnarray}
For the sake of brevity, we show the conclusion of the proof only for $n \geq 4$, since for $n=2,3$ the argument is the same. Inequality \eqref{fu} and Lemma \ref{utile} lead to
$$
\|u\|_{L^{\infty}(\S^{n-1})}^{n-1} \leq C(n)  \|D_\tau u\|_{L^{2}(\S^{n-1})}^{2} \| u\|_{L^{\infty}(\S^{n-1})}^\frac{n-3}{2},
$$
and hence we have
\begin{equation*}% \label{ultima}
    \|u\|_{L^{\infty}(\S^{n-1})}^\frac{n+1}{2} \leq C(n)  \|D_\tau u\|_{L^{2}(\S^{n-1})}^{2}.
\end{equation*}
Recalling now that $R_2\|u\|_{L^\infty}=d_\mathcal{H}(\Omega_0,B_{R_2})$, we get
$$\|\nabla_\tau u\|_{L^{2}(\S^{n-1})}^{2}\ge C_6(n,R_2)g(d_\mathcal{H}(\Omega_0,B_{R_2})).$$
Plugging the previous estimate in \eqref{robinnearly} and recalling that $\delta_0$ depends only on $n,\beta,R_1,R_2$, we finally obtain
\begin{align*}
\lambda_1(\beta,A_{R_1,R_2})-\lambda_1(\beta,\Omega)&\geq C_7(n,\beta,R_1,R_2,\delta_0)\|\nabla_\tau u\|_{L^2(\S^{n-1})}^2\geq C(n,\beta,R_1,R_2)g(\mathcal{A}_\mathcal{H}(\Omega_0)).
\end{align*}
\end{proof}

   % Specifically, by employing the same notation used throughout the paper, one aims the following kind of estimate:  

\subsection{Inner asymmetry}%\label{inner_sec}
In this section we explain the key idea to keep into account both inner and outer perturbations of the optimal set $A_{R_1,R_2}$. %In order to give an estimate of the asymmetry of $\Theta$, we use a completely different approach.
As previously explained, the Neumann condition on the inner boundary and the measure constraint on $\Omega=\Omega_0\setminus\overline{\Theta}$ do not seem to allow a Fuglede type approach, since we do not have any control on $d_{\mathcal{H}}(\Theta,B_{R_1})$. Intuitively, this is due to lack of a constraint on $P(\Theta)$ and to the fact that in the Rayleigh quotient \eqref{Raylegh_quotient}, there is no boundary integral on $\partial\Theta$. Moreover, if $\Omega$ is nearly annular and one tries to improve the quantitative result of Proposition \ref{nss}, for example by using the measure constraint \eqref{moreover_2}, only the terms involving the outer boundary profile $u$ remain and the ones containing the inner boundary profile $v$ cancel each other out.

By using the weak asymmetry relative to the outer box defined in \eqref{eq:atilde}, we prove the following result.

\begin{theorem}
\label{teo:buco}
Let $0<R_1<R_2<+\infty$. Then, for every $\Omega=\Omega_0\setminus\overline{\Theta}\subset\R^n$ with $\Omega_0$ convex and $\Theta$ an admissible hole such that $\overline{\Theta}\subset\Omega_0$, $P(\Omega_0)=P(B_{R_2})$ and $|\Omega|=|A_{R_1,R_2}|$,
it holds
\begin{equation*}
\lambda_1(\beta,A_{R_1,R_2})-\lambda_1(\beta,\Omega)\ge\min\{1,|\lambda_1(\beta,A_{R_1,R_2})|\}\tilde{\mathcal{A}}(\Theta;\Omega_0),%\int_{\Theta\setminus K}(z_{\Omega_0}^2-m_K^2)\:dx
\end{equation*}
where $\tilde{\mathcal{A}}(\Theta;\Omega_0)$ is defined in \eqref{eq:atilde}.
\end{theorem}
\begin{proof}
Let $z\in H^1(A_{R_1,R_2})$ be the first positive eigenfunction of $\lambda_1(\beta,A_{R_1,R_2})$ with $\|z\|_{L^2(A_{R_1,R_2})}=1$, let $w_{\Omega_0}\in H^1(\Omega_0)$  the test function defined in \eqref{GFerone} and $t_{\Omega_0},K_{\Omega_0}$ as in the Definition \ref{inner_asymmetry_def}. 

We observe that $t_{\Omega_0}\ge R_2-R_1$ in view of the isoperimetric inequality; it follows $w_{\Omega_0}\ge z_m$ in $\Omega_0\setminus\overline{ K_{\Omega_0}}$ and $w_{\Omega_0}=z_m$ in $K_{\Omega_0}$.%\textcolor{red}{with the equality holding if and only if $\Omega_0=B_{R_2}$ and consequently $K=B_{R_1}$. }

We point out that %provided that $\Omega_0$ is convex and nearly spherical, $\Omega_0\setminus\overline{K_{\Omega_0}}$ is an admissible set for Problem \eqref{robinproblem} and that 
the choice of $w_{\Omega_0}$ depends only on $\Omega_0$ and thus it is a suitable web function for both $\Omega_0\setminus\overline{\Theta}$ and $\Omega_0\setminus\overline{K_{\Omega_0}}$. From \cite[eq.(3.9),(3.10),(3.11)]{paoli2020sharp} applied to the admissible set $\Omega_0\setminus\overline{K_{\Omega_0}}$, we have that
\begin{equation}
    \label{eq:compnorm}
\begin{split}
\int_{\Omega_0\setminus K_{\Omega_0}}|\nabla w_{\Omega_0}|^2\:dx\le\int_{A_{R_1,R_2}}|\nabla z|^2\:dx,&\quad \int_{\partial\Omega_0}w^2_{\Omega_0}\:d\mathcal{H}^{n-1}=\int_{\partial B_{R_2}}z^2\:d\mathcal{H}^{n-1},\\
\int_{\Omega_0\setminus K_{\Omega_0}}w^2_{\Omega_0}\:dx&\le\int_{A_{R_1,R_2}}z^2\:dx.
\end{split}
\end{equation}
Moreover, it holds the following
\begin{equation}
\label{eq:theta-kappa}
    \begin{split}     
\int_{\Omega_0}w_{\Omega_0}^2\:dx&=\int_{\Omega_0\setminus \Theta}w_{\Omega_0}^2\:dx+\int_{\Theta\setminus K_{\Omega_0}}w_{\Omega_0}^2\:dx+\int_{K_{\Omega_0}\cap\Theta}w_{\Omega_0}^2\:dx\\
%&\le\int_{A_{R_1,R_2}}z^2\:dx+\int_{\Theta\setminus K}w_{\Omega_0}^2\:dx+\int_{K\cap\Theta}w_{\Omega_0}^2\:dx.
    \end{split}
\end{equation}
and, by using \eqref{eq:compnorm}, it holds
\begin{equation}\label{eq:full}
\begin{split}
\int_{\Omega_0}w_{\Omega_0}^2\:dx&=\int_{\Omega_0\setminus K_{\Omega_0}}w_{\Omega_0}^2\:dx+\int_{K_{\Omega_0}\setminus\Theta}w_{\Omega_0}^2\:dx+\int_{K_{\Omega_0}\cap\Theta}w_{\Omega_0}^2\:dx\\
&\le\int_{A_{R_1,R_2}}z^2\:dx+z_m^2|\Theta\setminus K_{\Omega_0}|+\int_{K_{\Omega_0}\cap\Theta}w_{\Omega_0}^2\:dx,
\end{split}
\end{equation}
since $|K_{\Omega_0}\setminus\Theta|=|\Theta\setminus K_{\Omega_0}|$, which follows from the fact that $|K_{\Omega_0}|=|\Theta|$. 
Therefore, by combining  \eqref{eq:theta-kappa} and \eqref{eq:full}, we get 
\begin{equation}  \label{eq:quantinorm}
\int_{\Omega_0\setminus \Theta}w^2_{\Omega_0}\:dx\le\int_{A_{R_1,R_2}}z^2\:dx-\int_{\Theta\setminus K_{\Omega_0}}(w_{\Omega_0}^2-z_m^2)\:dx.
\end{equation}

Now, by splitting in the same way also $\int_{\Omega_0}|\nabla w_{\Omega_0}|^2\:dx$ and recalling that $w_{\Omega_0}$ is constant in $K_{\Omega_0}$, we get

\begin{equation}\label{split_grad_1}
\begin{split}     
\int_{\Omega_0}|\nabla w_{\Omega_0}|^2\:dx&=\int_{\Omega_0\setminus \Theta}|\nabla w_{\Omega_0}|^2\:dx+\int_{\Theta\setminus K_{\Omega_0}}|\nabla w_{\Omega_0}|^2\:dx+\int_{K_{\Omega_0}\cap\Theta}|\nabla w_{\Omega_0}|^2\:dx\\
&=\int_{\Omega_0\setminus \Theta}|\nabla w_{\Omega_0}|^2\:dx+\int_{\Theta\setminus K_{\Omega_0}}|\nabla w_{\Omega_0}|^2\:dx
\end{split}
\end{equation}
and, in view of \eqref{eq:compnorm},
\begin{equation}\label{split_grad_2}
\begin{split}
\int_{\Omega_0}|\nabla w_{\Omega_0}|^2\:dx&=\int_{\Omega_0\setminus K_{\Omega_0}}|\nabla w_{\Omega_0}|^2\:dx+\int_{K_{\Omega_0}}|\nabla w_{\Omega_0}|^2\:dx\le\int_{A_{R_1,R_2}}|\nabla z|^2\:dx.
\end{split}
\end{equation}
By combining \eqref{split_grad_1} and \eqref{split_grad_2} we get
\begin{equation}  \label{eq:quantigrad}
\int_{\Omega_0\setminus \Theta}|\nabla w_{\Omega_0}|^2\:dx\le\int_{A_{R_1,R_2}}|\nabla z|^2\:dx-\int_{\Theta\setminus K_{\Omega_0}}|\nabla w_{\Omega_0}|^2\:dx.
\end{equation}

Now, by using $w_{\Omega_0}$ as a test function for $\lambda_1(\beta,\Omega)$, we have, in view of \eqref{eq:compnorm}, \eqref{eq:quantinorm} and \eqref{eq:quantigrad}
\begin{align*}
\lambda_1(\beta,\Omega)&\le\frac{\displaystyle\int_{\Omega_0\setminus \Theta}|\nabla w_{\Omega_0}|^2\:dx+\beta\int_{\partial\Omega_0}w^2_{\Omega_0}\:d\mathcal{H}^{n-1}-\int_{\Theta\setminus K_{\Omega_0}}|\nabla w_{\Omega_0}|^2\:dx}{\displaystyle\int_{\Omega_0\setminus \Theta}w^2_{\Omega_0}\:dx}\\
&\le\frac{\displaystyle\int_{A_{R_1,R_2}}|\nabla z|^2\:dx+\beta\int_{\partial A_{R_1,R_2}}z^2\:d\mathcal{H}^{n-1}-\int_{\Theta\setminus K_{\Omega_0}}|\nabla w_{\Omega_0}|^2\:dx}{\displaystyle\int_{A_{R_1,R_2}}z^2\:dx-\int_{\Theta\setminus K_{\Omega_0}}(w_{\Omega_0}^2-z_m^2)\:dx}\\
&=\frac{\displaystyle\lambda_1(\beta,A_{R_1,R_2})-\int_{\Theta\setminus K_{\Omega_0}}|\nabla w_{\Omega_0}|^2\:dx}{\displaystyle 1-%\frac{1}{\|z\|_{L^2(A_{R_1,R_2})}^2}
\int_{\Theta\setminus K_{\Omega_0}}(w_{\Omega_0}^2-z_m^2)\:dx},
\end{align*}
and thus, from $-\lambda_1(\beta,\Omega)\ge-\lambda_1(\beta,A_{R_1,R_2})=|\lambda_1(\beta,A_{R_1,R_2})|$, we finally obtain
\begin{align*}
\lambda_1(\beta,A_{R_1,R_2})-\lambda_1(\beta,\Omega)&\ge-
\lambda_1(\beta,\Omega)
\int_{\Theta\setminus K_{\Omega_0}}(w_{\Omega_0}^2-z_m^2)\:dx+\int_{\Theta\setminus K_{\Omega_0}}|\nabla w_{\Omega_0}|^2\:dx\\
&\ge
|\lambda_1(\beta,A_{R_1,R_2})|
\int_{\Theta\setminus K_{\Omega_0}}(w_{\Omega_0}^2-z_m^2)\:dx+\int_{\Theta\setminus K_{\Omega_0}}|\nabla w_{\Omega_0}|^2\:dx,
\end{align*}
achieving the thesis.
\end{proof}

%We remark that the previous theorem provides a suitable constant for our purposes: the dependence on $\Omega_0$ and $\Theta$ appears only in the term $\tilde{\mathcal A}(\Theta;\Omega_0)$, while the multiplicative constant $|\lambda_1(\beta,A_{R_1,R_2})|$ depends only on the parameters of the problem $n,\beta,R_1,R_2$. 

\subsection{Towards the nearly annular sets}

We want to show that the stability result in the class $\mathcal{T}_{R_1,R_2}$, defined in \eqref{admissiblesets}, gets meaningful if we reduce to nearly annular sets. To do that, the first step is to provide a uniform bound on the diameters of the sets  whose eigenvalues are "not so far" from the optimum.  This kind of isodiametric control of the eigenvalues is common when dealing with maximization problems in shape optimization and it often provides extra compactness when working on existence problems without a bounded design region (see for instance the isodiametric control of the Robin spectrum proved in \cite{bucurcito} or for the Steklov spectrum proved in \cite{bbg}, both valid also for higher eigenvalues in a wider class of sets).

%Before proving the result, we point out that this isodiametric control holds also for holed domains which are not in the class $\mathcal{T}_{R_1,R_2}$, defined in \eqref{admissiblesets}, so we give the proof in this general case (in the class $\mathcal{T}_{R_1,R_2}$ the proof is immediate).

\begin{lemma}\label{isod}
Let $0<R_1<R_2<+\infty$. There exists a positive constant $C(R_1,R_2,n,\beta)$ such that, for every $\Omega=\Omega_0\setminus\overline{\Theta}\subset\R^n$ with $\Omega_0$ and $\Theta$ convex sets such that $\overline{\Theta}\subset\Omega_0$, $P(\Omega_0)=P(B_{R_2})$ and $|\Omega|=|A_{R_1,R_2}|$, if $$\lambda_1({\beta},\Omega)>2 \lambda_1({\beta},A_{R_1,R_2})$$
then
$$\text{diam}(\Omega)<C(R_1,R_2,n,\beta).$$
\begin{proof}
The proof is a straightforward adaptation to our context of \cite[Lemma 3.5]{cito2021quantitative}.

Let us argue by contradiction and suppose that there exists a sequence $\{\Omega_j\}_{j\in\N}$ of domains of the form $\Omega_j={\Omega_0}_j\setminus\overline{\Theta_j}$ such that 
$$\lambda_1({\beta},A_{R_1,R_2})\ge\lambda_1({\beta},\Omega_j)>2 \lambda_1({\beta},A_{R_1,R_2})\ \text{and}\ \text{diam}(\Omega_j)=\text{diam}({\Omega_0}_j)\to+\infty.$$
In view of the convexity of ${\Omega_0}_j$ and of the constraint $P({\Omega_0}_j)=m$, the sequence $\{\rho_j\}_{j\in\N}$ of the inradii of $\{{\Omega_0}_j\}_{j\in\N}$ is necessarily vanishing. Recalling that, for any convex set $A$ with inradius $\rho$, it holds
$$|A|\le\rho P(A)$$
(see, for instance, \cite[Prop. 2.4.3]{bucur2004variational}), we deduce that $|{\Omega_0}_j|$ vanishes as $j$ goes to $+\infty$. Now, using the charachteristic function $\chi_{\Omega_j}$ as a test for $\lambda_1({\beta},\Omega_j)$, we obtain
$$\lambda_1({\beta},\Omega_j)\le-\beta\frac{P({\Omega_0}_j)}{|\Omega_j|}\le-\beta\frac{P({\Omega_0}_j)}{{|{\Omega_0}_j}|}\to-\infty,$$
in contradiction with the lower bound on $\lambda_1({\beta},\Omega_j)$.
\end{proof}
\end{lemma}

We point out that the previous result is immediate if we restrict to the class $\mathcal{T}_{R_1,R_2}$. 

Now, we provide two useful semicontinuity results: the first one gives the lower semicontinuity of the boundary integral term and the second one give the upper semicontinuity of the map $\Omega\mapsto\lambda_1(\beta,\Omega)$.
\begin{lemma}[\cite{citoconvex}, Proposition 2.1]\label{pro_lscquadro}
Let $E_j,E\subset\R^n$ be convex domains such that $E_j\to E$ in the sense of Hausdorff and in measure. Let $w_j,w\in H^1(\R^n)$. If $w_j\rightharpoonup w$ in $H^1(\R^n)$, then it holds
$$\int_{\partial E} w^2\:d\sigma\le \liminf_{j}\int_{\partial E_j} w_j^2\:d\sigma$$
\end{lemma}

%The following upper semicontinuity issue holds for the functional $\Omega\mapsto\lambda_1(\beta,\Omega)$.
\begin{lemma}\label{usc_lambda}
Let $\Omega_j,\Omega\subset \R^n$ be in the class $\mathcal{T}_{R_1,R_2}$,  with $\Omega_j \to \Omega$ in the sense of Hausdorff and in measure. Then 
$$
\limsup_{j\to+\infty} \lambda_1(\beta,\Omega_j)  \leq\lambda_1(\beta,\Omega).
$$
\begin{proof}
We proceed similarly to \cite[Prop. 3.1]{citoconvex}, being careful to the fact that the boundary integral keeps into account only $\partial\Omega_0$. Let $u\in H^1(\Omega)$ an eigenfunction for $\lambda_1(\beta,\Omega)$, and let $\overline{u}\in H^1(\R^n)$ an extension of $u$ to the whole of $\R^n$. Let us notice that, in view of the hypotheses, the convergence of the convex sets ${\Omega_0}_j \to \Omega_0$ in the sense of Hausdorff and in measure holds. Thus we have
$$
\int_{\partial \Omega_0}u^2 d\H^{n-1}=\int_{\partial \Omega_0}\overline{u}^2 d\H^{n-1}\le\liminf_{j\to+\infty}\int_{\partial {\Omega_0}_j}\overline{u}^2 d\H^{n-1}
$$
as a consequence of Lemma \ref{pro_lscquadro}. Moreover, both volume integrals are continuous in view of the convergence in measure $\Omega_j\to\Omega$ and this implies the upper semicontinuity of the Rayleigh quotients, defined in \eqref{Raylegh_quotient}. Therefore, we have
$$\lambda_1({\beta},\Omega)=R_\Omega(u)=R_\Omega(\overline{u})\ge\limsup_{j\to+\infty}R_{\Omega_j}(\overline{u})\ge\limsup_{j\to+\infty} \lambda_1(\beta,\Omega_j).$$
\end{proof}
\end{lemma}

Since $A_{R_1,R_2}$ is the unique maximizer of $\lambda_1(\beta,\cdot)$ in our class of admissible sets, as a consequence of the upper semicontinuity of $\lambda_1(\beta,\Omega)$ and of the isodiametric control in Lemma \ref{isod}, the following convergence result holds.

\begin{lemma}\label{lem:maxi}
Let $\{\Omega_j\}_{j\in\N}\subset\mathcal{T}_{R_1,R_2}$ be a maximizing sequence for Problem \eqref{robinproblem} with $\Omega_j$ is barycentered at the origin. Then  $$d_\H(\Omega_j,A_{R_1,R_2}) \to 0.$$
\begin{proof}
Thanks to Lemma \ref{isod}, the diameters of the convex sets ${\Omega_0}_j$ and $\Theta_j$ are uniformly bounded. 

%Let us suppose that $\Omega_{0,j}$ does not converge to $B_{R_2}$. Thus, there exists $\bar{\varepsilon}$ such that (possibly passing to a subsequence) $$\liminf_j d_\mathcal{H}(\Omega_{0,j},B_{R_2})>\bar{\varepsilon}.$$
%As a consequence, we get that $|\Omega_j\setminus A_{R_1,R_2}|$ does not vanish. Otherwise, since $diam(\Omega_j)<C$, $P(\Omega_{0,j})=n\omega_nR_2^n$ and $d_\mathcal{H}(\Omega_{0,j},B_{R_2})$ is bounded from below by a positive constant, we should have that $|\Omega_{0,j}|\to 0$, a contradiction. Let now $z$ be an eigenfunction for $A_{R_1,R_2}$ such that $\|u\|_{L^2(A_{R_1,R_2})}=1$ and let us denote 

%Since $|\Omega_{0,j}|\to|B_{R_2}|$, then $|\Omega_{0,j}|=|B_{R_2}|-o(1)$. As a consequence we have that $$|\Theta_j|=|\Omega_{0,j}|-|\Omega_j|=|B_{R_2}|-o(1)-|A_{R_1,R_2}|=|B_{R_1}|-o(1),$$then $|\Theta_j|$ is uniformly bounded from below.
The uniform upper bound on the diameters and the uniform lower bound $\rho(\Theta_j)>\vartheta_{R_1,R_2}$ imply that, up to subsequences, there exist an open bounded convex set $\Omega_0$ and an open bounded convex set $\Theta\subset\subset\Omega_0$ such that $\Omega:=\Omega_0\setminus\overline{\Theta}\in\mathcal{T}_{R_1,R_2}$ and $d_\H(\Omega_{j},\Omega) \to 0$. Since $\lambda_1(\beta,\cdot)$ is upper semicontinuous in view of Lemma \ref{usc_lambda}, we obtain that
$$
\max_{\mathcal{A}_{R_1,R_2}}\lambda_1(\beta,\cdot)=\limsup_{j\to+\infty}\lambda_1(\beta,\Omega_j)  \leq\lambda_1(\beta,\Omega).
$$
So, since the spherical shell $A_{R_1,R_2}$ is the unique maximizer for $\lambda_1(\beta,\cdot)$ in $\mathcal{T}_{R_1,R_2}$, we get that $\Omega=A_{R_1,R_2}$  necessarily.
\end{proof}
\end{lemma}

As a consequence of the previous lemma, we can actually restrict our main stability result to nearly spherical sets barycentered in the origin.
\begin{lemma} \label{red}
Let $0<R_1<R_2<+\infty$. There exists a positive constant $\delta_0=\delta_0(n,\beta,R_1,R_2)$ such that, if $\Omega\in\mathcal{T}_{R_1,R_2}$ and  $$\lambda_1(\beta,A_{R_1,R_2})-\lambda_1(\beta,\Omega) \leq \delta_0$$
then, up to a translation, $\Omega$ is a $(R_1,R_2)-$nearly annular set.
\end{lemma}
\begin{proof}
The proof is based on a straightforward contradiction argument. Indeed, let us suppose that, for every $j\in\N$, there exists an admissible set $\Omega_j={\Omega_0}_j\setminus\overline{\Theta_j}$ such that
\begin{equation}\label{5.5.1}
\lambda_1(\beta,A_{R_1,R_2})-\lambda_1(\beta,\Omega_j) \leq \frac{1}{j}
\end{equation}
and any translations of $\Omega_j$ is not a $(R_1,R_2)-$nearly annular set, i.e.
\begin{equation}\label{5.5.2}
\text{either}\quad d_\mathcal{H}({\Omega_0}_j,B_{R_2})>\frac{R_2}{2}\quad \text{or} \quad d_\mathcal{H}(\Theta_j,B_{R_1})>\frac{R_1}{2}.
\end{equation}

By \eqref{5.5.1}, $\{\Omega_j\}_{j\in\N}$ is a maximizing sequence for Problem \eqref{robinproblem}; if we consider the sequence (still denoted by $\{\Omega_j\}_{j\in\N}$) of translated sets with barycenter at the origin, we are under the hypotheses of Lemma \ref{lem:maxi} and thus we get $d_\mathcal{H}(\Omega_j,A_{R_1,R_2})\to 0$, that is in contradiction with \eqref{5.5.2}.
\end{proof}

\subsection{Conclusion}

Now we are in position to prove both the Main Theorems stated in the introduction. 

\begin{proof}[Proof of the Main Theorem 1]
It is a straightforward combination of Theorem \ref{teo:esterno} and Theorem \ref{teo:buco}.
\end{proof}

\begin{proof}[Proof of the Main Theorem 2]
It is a straightforward combination of Theorem \ref{teo:esterno}, Theorem \ref{teo:buco} and Lemma \ref{red}, which ensures us that, if $\delta_0$ is small enough, then we can suppose without loss of generality that $\Omega$ is a nearly $(R_1,R_2)-$annular set with barycenter at the origin, $P(\Omega_0)=n\omega_n R_2^{n-1}=P(B_{R_2})$ and $|\Omega|=A_{R_1,R_2}$.
\end{proof}

\begin{remark}
We point out that, as usual when dealing with negative Robin boundary conditions, the constant $C(n,R_1,R_2,\beta)$ appearing in the stability inequalities \eqref{stab_ineq_intro1} and \eqref{stab_ineq_intro2}
do not depend only on the dimension, but also on the boundary parameter $\beta$ and on $R_1$ and $R_2$, i.e. on the size of the admissible sets for the maximization problem. %This is a standard behaviour when dealing with the stability of a spectral inequality with Robin boundary conditions and negative boundary parameter, because of the  .
\end{remark}

\section{Open problems}
\label{rem_sec}

In this Section, we collect some open problems.

\begin{open}
    By repeating the very same computations as in Theorem \ref{teo:buco} for the $p$-Laplace Robin-Neumann eigenvalue $\lambda_1^{(p)}(\beta,\Omega)$, we get
$$\lambda^{(p)}_1(\beta,A_{R_1,R_2})-\lambda_1^{(p)}(\beta,\Omega)\ge
\min\{1,|\lambda^{(p)}_1(\beta,A_{R_1,R_2})|\}
\int_{\Theta\setminus K}(|\nabla w_{\Omega_0}|^p+w_{\Omega_0}^p-z_m^p)\:dx.$$
To get a complete stability result also for the nonlinear case we need an estimate for the asymmetry of $\Omega_0$ which at the moment does not seem available, since our argument based on a Fuglede type approach does not seem to apply in the nonlinear framework due to the lack of an explicit expression for $z$.

Moreover, we ask whether the method used in this paper may be applied in more general frameworks, as e.g. the Minkowski spaces, in which the standard Euclidean norm is replaced by a Finsler norm \cite{van2006anisotropic}. 
As far as Robin boundary conditions in the anisotropic setting, we refer e.g to \cite{paoli2019two,  barbato2023first}, where isoperimetric inequalities, involving possibly nonlinear elliptic operators, are proved. It remains open the problem of the generalisation to holed domains with Robin- Neumann anisotropic boundary conditions and the relative quantitative estimates. 
\end{open}

\begin{open}
It would be interesting to consider the case $\beta>0$. In \cite{paoli2020sharp}, it is proved that  the spherical shell such that $|A|=|\Omega|$ and  $P(B_{R_2})=P(\Omega_0)$ is maximizing:
\begin{gather*}
\lambda_1%^{RN}
(\beta, \Omega)\leq \lambda_1%^{RN}
(\beta, A).
\end{gather*}
A stability result in this sense requires different methods because, for example, it is not possible to consider the auxiliary Steklov-type problem. %The stability result could rely on the quantification of the Alexandrov-Fenchel inequalities. %The objective is to show quantitatively how much $\Omega$ is close to the spherical shell.

\begin{open}
Prove the analogous of the isoperimetric inequality \eqref{ineq_eig}, contained in \cite{paoli2020sharp}, in the case when $\beta$ is a positive or negative function satisfying suitable regularity assumptions, and prove the associated quantitative inequality.
\end{open}

\begin{open}
It remains open the problem of proving the sharpness of the our quantitative results.

%It is possible to prove the result of the Main Theorem also in some larger classes than $\mathcal{T}_{R_1,R_2}$. For instance, one can consider $\Theta$ as a uniformly finite union of nondegenerate convex sets such that the connected components of $\partial\Omega$ lie at uniformly bounded distance and the inradii are uniformly bounded from below. More generally, one can replace the convexity constraint and the nondegeneracy assumptions on the admissible sets $\Theta$ by a uniform cone condition. 

%Indeed, the isoperimetric inequality \eqref{ineq_eig} has been proved in \cite{paoli2020sharp} in a larger setting than $\mathcal{T}_{R_1,R_2}$, and so the proof of the stability \eqref{stab_ineq_intro1} and \eqref{stab_ineq_intro2} can be easily generalized to the above mentioned more general cases.
\end{open}
\end{open}

\section*{Acknowledgment}
S. Cito was supported by "Elliptic and parabolic problems, heat kernel estimates and spectral theory" PRIN 2022 project n. 20223L2NWK, funded by the Italian Ministry of University and
Research.

G. Paoli was supported by "A sustainable and trusted Transfer Learning Platform for Edge Intelligence" PRIN 2022 PNRR project n. P2022TTW7L, funded by
the Italian Ministry of University and
Research.

G. Piscitelli was supported by “Geometric-Analytic Methods for PDEs and Applications (GAMPA)” PRIN 2022 project n. 2022SLTHCE, funded by Italian Ministry of University and
Research.

The three authors was supported by Gruppo Nazionale per l’Analisi
Matematica, la Probabilità e le loro Applicazioni (GNAMPA) of Istituto Nazionale di Alta
Matematica (INdAM).

 G. Paoli was supported by the Alexander von Humboldt Foundation through an Alexander von Humboldt research fellowship. 
 
\bibliographystyle
%{alpha}
{abbrv}
\bibliography{bibliography.bib}
\end{document}